\documentclass[12pt]{amsart}
\usepackage{graphicx}
\usepackage[mathscr]{eucal}
\usepackage{microtype}
\usepackage{amssymb}
\usepackage[parfill]{parskip}
\usepackage[dvipsnames]{xcolor}
\usepackage[top=1in, bottom=1in, left=1.5in, right=1.5in]{geometry}
\usepackage{enumerate}
\usepackage{hyperref,cite}
\usepackage{bm}
\usepackage[normalem]{ulem}
\usepackage{mathrsfs}

\newcommand{\br}{\mathbb{R}}

\newcommand{\bn}{\mathbb N}
\newcommand{\bl}{\mathbb L}
\newcommand{\bS}{\mathbb S}

\newcommand{\al}{\alpha}
\newcommand{\be}{\beta}

\newcommand{\vp}{\varphi}

\newcommand{\ce}{\mathcal E}

\newcommand{\cf}{\mathcal F}

\newcommand{\ssm}{\smallsetminus}

\newcommand{\co}{\colon\thinspace}

\DeclareMathOperator{\cl}{cl}
\DeclareMathOperator{\h}{ht}
\renewcommand{\int}{\mathop{\mathrm{int}}}

\newtheorem{Thm}{Theorem}[section]
\newtheorem{Prop}[Thm]{Proposition}
\newtheorem{Lem}[Thm]{Lemma}
\newtheorem{Cor}[Thm]{Corollary}

\theoremstyle{definition}
\newtheorem{Def}[Thm]{Definition}

\theoremstyle{remark}
\newtheorem{Rem}[Thm]{Remark}
\newtheorem{Ex}[Thm]{Example}
\newtheorem{Notation}[Thm]{Notation}

\numberwithin{equation}{section}

\makeatletter
\let\@wraptoccontribs\wraptoccontribs
\makeatother

\begin{document}

\title[Non-metrizable manifolds: ends and a general bagpipe theorem]{Ends of non-metrizable manifolds: \\
a generalized bagpipe theorem}

\author[D. Fern\'andez]{David Fern\'andez-Bret\'on}
\address{Departamento de Matem\'aticas\\
Centro de Investigaci\'on y Estudios Avanzados.
}

\curraddr{Escuela Superior de F\'{\i}sica y Matem\'aticas\\
Instituto Polit\'ecnico Nacional\\
Col. San Pedro Zacatenco, Alcald\'{\i}a Gustavo A. Madero, 07738, CDMX, Mexico. 
}
\email{dfernandezb@ipn.mx 
}

\author[N. G. Vlamis]{Nicholas G. Vlamis}
\address{Department of Mathematics\\
Queens College, City University of New York\\
65-30 Kissena Blvd., Queens, NY 11367, U.S.A.
}
\email{nicholas.vlamis@qc.cuny.edu 
}

\contrib[An Appendix with]{Mathieu Baillif}


\date{\today}


\begin{abstract}
We initiate the study of ends of non-metrizable manifolds and introduce the notion of short and long ends.
Using the theory developed, we provide a characterization of (non-metrizable) surfaces that can be written as the topological sum of a metrizable manifold plus a countable number of ``long pipes" in terms of their spaces of ends; this is a direct generalization of Nyikos's bagpipe theorem.
\end{abstract}

\maketitle


\section{Introduction}

An $n$-manifold is a connected Hausdorff topological space that is locally homeomorphic to $\mathbb R^n$.
Often---especially in geometric and low-dimensional topology---second countability is also included as part of the definition; however, many more possibilities for manifolds arise when second countability is not required.
Manifolds that fail to be second countable are generally referred to as {\it non-metrizable manifolds}.
There has been much work devoted to understanding their structure from a set-theoretic  viewpoint~\cite{NyikosTheory,GauldNonmetrisable,GreenwoodThesis,GreenwoodTrees,GauldMappingClass}.

One of the key results regarding the study of non-metrizable $2$-manifolds is Nyikos's {\it bagpipe theorem}, characterizing ``bagpipes'', that is, a $2$-manifolds that can be decomposed into a compact manifold---the bag---plus a finite number of ``long pipes''. The main purpose of this paper is to offer a characterization of a broader class of $2$-manifolds, which we call ``general bagpipes''. 
Informally, a general bagpipe is a 2-manifold that can be decomposed into a metrizable manifold---a bigger bag---plus countably many ``long pipes". 
Our characterization combines the set-theoretic viewpoint of Nyikos (and others)  together with a tool commonly used in low-dimensional topology---the ends of a manifold,  which describe the possible ways in which one can ``go off to infinity'' within the manifold. We are able to describe various classes of manifolds in terms of their ends, culminating with a precise description of the structure of ends of general bagpipes.

\begin{Def}
Let\footnote{Throughout this paper, the set of natural numbers starts at $1$, so $\mathbb N=\{1,2,3,\ldots\}$.} $n\in\mathbb N$. A connected Hausdorff topological space $M$ is an {\bf $n$-manifold} if for every $x\in M$ there exists an open neighborhood $U$ of \( x \) that is homeomorphic to $\mathbb R^n$. A {\bf manifold} is a space that is an $n$-manifold for some $n$.
\end{Def}

Thus, throughout this paper we use the word ``manifold'' in its most general sense, without the second-countability restriction. It is well-known that a manifold is metrizable if and only if it is second countable, which is in turn the case if and only if it is Lindel\"of, and either of these is also equivalent to paracompactness (see~\cite[2.2, pp. 27--36]{GauldNonmetrisable} for a list of 119 properties that are equivalent to metrizability of a manifold). Hence, we will speak about metrizable and non-metrizable manifolds, accordingly; simply using the word {\it manifold} will not imply any assumption one way or the other. 

A very immediate example of a non-metrizable manifold is the long ray $\mathbb L^+$, defined as the linearly ordered topological space given by the product $\omega_1\times[0,1)$ (where $\omega_1$ denotes the first uncountable ordinal), equipped with the lexicographical order, and with the minimum point deleted. This topological space can be thought of, intuitively, as the result of gluing together $\omega_1$ many copies of the unit interval; clearly the non-metrizability of the long ray is related to the fact that $\omega_1$ is ``too long''. Nyikos introduced the concept of a {\it type I} manifold to formalize the intuition of a manifold that lacks metrizability because of its being ``too long'', though there are other examples of non-metrizable manifolds (for instance, the Pr\"ufer manifold, see Appendix~\ref{appendix-prufer}) that are non-metrizable for other reasons. A manifold $M$ is of type I if it can be written as $M=\bigcup_{\alpha<\omega_1}M_\alpha$, where each $M_\alpha$ is a metrizable open subspace such that $\cl_M(M_\alpha)\subseteq M_{\alpha+1}$ for all $\alpha<\omega_1$. 
Outside of dimension 1, the simplest examples of type I manifolds are long planes: a \emph{long plane} is a manifold $P$ which can be written as $P=\bigcup_{\alpha<\omega_1}M_\alpha$ with each $M_\alpha$ homeomorphic to $\mathbb R^2$ and the boundary of each $\cl_M(M_\alpha)$ is contained in $M_{\alpha+1}$ and is homeomorphic to the unit circle $\mathbb S^1$. 
A \emph{long pipe}\footnote{This definition of long pipe is slightly more restrictive that Nyikos's original definition.  See Definition~\ref{def:pipe} for a discussion of the definition of long pipe.} is a manifold obtained from deleting a point from a long plane; the simplest example of a long pipe is the product $ \mathbb S^1\times\mathbb L^+$, with $\mathbb L^+$ the long ray; there are, however, uncountably many  non-homeomorphic spaces that satisfy the definition of a  long pipe (see~\cite[p. 662 and Section 6]{NyikosTheory}).

\begin{Thm}[``The Bagpipe Theorem", Nyikos~\cite{NyikosTheory}]
Let $M$ be a $2$-manifold.
The closure of every countable subset of $M$ is compact if and only if there exist finitely many pairwise-disjoint embedded long pipes \( P_1, \ldots, P_n \) in \( M \) such that the complement of \( \bigcup_{j=1}^n P_j \) in \( M \) is compact. 
\end{Thm}

The 2-manifolds satisfying the conditions of the Bagpipe Theorem are called \emph{bagpipes} and are the non-metrizable generalization of a compact 2-manifold.

An important tool for the study of non-compact metrizable manifolds are the so-called {\it ends} of a manifold $M$. Ends are elements of the remainder $\mathcal F(M)\setminus M$ of the {\it Freudenthal compactification} $\mathcal F(M)$ of the manifold $M$, which---intuitively speaking---constitutes a way of adding points at infinity that the manifold $M$ is missing (formal definitions will be stated in Section~\ref{sec:ends}). Although ends of a metrizable manifold are more or less understood, the authors are unaware of a study of ends of non-metrizable manifolds; we initiate this investigation here and develop some structure theory for end spaces of non-metrizable manifolds.

One interesting aspect is how individual ends themselves might sit within the Freudenthal compactification of the manifold, especially in ways that are not seen in the metrizable case. This leads to the definition of a {\it short end} and a {\it long end}, trying to capture the fact that some ends might be reached after an infinite, but countable, amount of time, while others might require $\omega_1$-many units of time to be reached (once again, formal definitions will be stated in Section~\ref{sec:ends}); thus, every end of a metrizable manifold is short, whereas, for example, long planes have a unique end, which is long. Surprisingly, some non-metrizable manifolds can have ends that are neither short nor long. This forces us to restrict our study precisely to those manifolds for which every end is either short or long; these manifolds will be said to have the {\it end dichotomy property}, or \emph{EDP} for short.

The structure of the space of ends of a non-metrizable manifold turns out to provide very detailed information about the manifold itself. In fact, we prove:

\begin{Thm}
\label{mainthm1}
Let \( M \) be a type I manifold.
\( M \) is metrizable if and only if the end space \( \cf(M) \ssm M \) of \( M \) is second countable and every end of \( M \) is short. 
\end{Thm}

For the sake of clarity, here in the introduction we will state all of our results in terms of type I manifolds; however, our theorems---and their proofs---are slightly more general.

The characterization of metrizability from Theorem~\ref{mainthm1} is crucial to the proof of the main theorem of this paper, which we state next.

\begin{Thm}[``The General Bagpipe Theorem"]
\label{mainthm2}
Suppose that $M$ is a 2-manifold of type I.
The space of ends \( \mathcal F(M) \ssm M \) of \( M \) is second countable and \( M \) has the end dichotomy property if and only if there exist at most countably many pairwise-disjoint embedded long pipes \( \{P_i\}_{i\in I} \) in \( M \) such that the complement of \( \bigcup_{i\in I} P_i \) in \( M \) is Lindel\"of. 
\end{Thm}

We call the 2-manifolds satisfying the conditions of the General Bagpipe Theorem \emph{general bagpipes}. 
We will see that both the end dichotomy property and the second countability of the space of ends are necessary by providing examples of 2-manifolds of type I that are not general bagpipes yet they satisfy one of these two conditions.

The heavy lifting in the General Bagpipe Theorem is a theorem about manifolds of arbitrary dimension that itself generalizes a result of Nyikos, the Bagpipe Lemma \cite[Theorem 5.9]{NyikosTheory}.
Before stating his result, we need a definition: 
A locally connected space \( M \) is \emph{trunklike} if, given any closed Lindel\"of subset \( C \), the set \( M \ssm C \) has at most one non-Lindel\"of component. 

\begin{Thm}[``The Bagpipe Lemma'' {\cite[Theorem 5.9]{NyikosTheory}}]
Let \( M \) be a manifold such that the closure of every countable subset of \( M \) is compact. Then, there exist finitely many pairwise-disjoint embedded trunklike manifolds \( T_1, \ldots, T_n \) in \( M \) such that the complement of \( \bigcup_{j=1}^n T_j \) in \( M \) is Lindel\"of.
\end{Thm}

Again, we generalize by allowing countably many trunklike manifolds in the decomposition.
We define a \emph{long trunk} to be a type I trunklike manifold with the end dichotomy property and finitely many ends, exactly one of which is long. 

\begin{Thm}[``The General Bagpipe Lemma"]
\label{mainthm3}
Let \( M \) be a  type I manifold.
The space of ends \( \mathcal F(M) \ssm M \) of \( M \) is second countable and \( M \) has the end dichotomy property if and only if there exist at most countably many pairwise-disjoint embedded long trunks \( \{T_i\}_{i \in I} \) in \( M \) such that the complement of the closure of \( \bigcup_{i\in I} T_i \) in \( M \) is Lindel\"of.
\end{Thm}

By restricting either the General Bagpipe Lemma or the General Bagpipe Theorem to manifolds in which every countable set has compact closure, we obtain yet another characterization of Nyikos's bagpipes, which are precisely those 2-manifolds that have a finite number of ends, all of which are long (see Theorem~\ref{lem:finite-long}). 

Finally, as noted, the proofs of these theorems below yield slightly more general results than stated here.  
In particular, the forward implications of Theorem~\ref{mainthm2} and Theorem~\ref{mainthm3} still hold if \( M \) is no longer required to be of type I, but its end space is required to be countable (instead of second countable).

We will see in the appendices that these assumptions cannot be weakened further:
In Appendix~\ref{sectionexamples}, we construct an example of a type I manifold, with all ends short, that is not metrizable---illustrating that the hypothesis that the end space is second countable in Theorem~\ref{mainthm2} and Theorem~\ref{mainthm3} is necessary.
In Appendix~\ref{appendix-prufer}, we construct an example of a non-metrizable manifold, not of type I, with all ends short and second-countable end space---illustrating that the hypothesis that the manifold be of type I in Theorem~\ref{mainthm2} and Theorem~\ref{mainthm3} is necessary\footnote{An earlier version of this paper conjectured that the type I assumption could be dropped in Theorem~\ref{mainthm2} and Theorem~\ref{mainthm3}, but M. Baillif provided a counterexample, which is now the core content of Appendix~\ref{appendix-prufer}.}.

This paper's high-level structure is as follows: in Section~\ref{sec:ends} we proceed to state a few basic definitions regarding the Freudenthal compactification, as well as the definitions of long and short ends.
In Section~\ref{sec:criterion}, we provide a criterion for metrizability of a manifold in terms of its Freudenthal compactification, proving Theorem~\ref{mainthm1} as Theorem~\ref{thm:characterize}.
In Section~\ref{sec:bagpipe} we prove Theorem~\ref{mainthm2} (see Theorem~\ref{thm:big-bag}) and Theorem~\ref{mainthm3} (see Theorem~\ref{thm:gen-bag-lemma}) generalizing Nyikos's Bagpipe Theorem and Bagpipe Lemma, respectively.
The paper ends with two appendices described in the previous paragraph.

\section*{Acknowledgments}

This collaboration began while the two authors were postdoctoral fellows at the University of Michigan (during this time, the second author was partially supported by NSF RTG grant 1045119). Afterwards, the first author was partially supported by the Consejo Nacional de Ciencia y Tecnolog\'{\i}a (grant FORDECYT \# 265667), and the second author received support from PSC-CUNY grant 62571-00 50. Thanks are also due to Rodrigo Hern\'andez-Guti\'errez for pointing us to several useful references, to David Gauld for making us aware of some inaccuracies in an earlier version of this paper, as well as to the anonymous referee for a nontrivial number of nontrivial improvements.

\section{Ends and the Freudenthal compactification}
\label{sec:ends}

The classification of non-compact second-countable 2-manifolds relies on the notion of a topological end, which codifies the idea of escaping to infinity in a topological space.
The general theme of this paper is exploring topological ends in non-metrizable manifolds and understanding to what extent the information about the structure at infinity controls the overall topological structure of the underlying manifold.

\subsection{The Freudenthal compactification}
In a search for invariants of topological groups, Freudenthal~\cite{Freudenthal} was among the first to rigorously define a topological end.
The ends can be used to compactify a space, obtaining what is referred to as the \emph{Freudenthal compactification}. When working with a hemicompact locally compact topological space\footnote{A topological space $X$ is hemicompact if there is a countable sequence $\langle K_n\big|n\in\mathbb N\rangle$ of compact subsets such that for every compact $K\subseteq X$ there is an $n\in\mathbb N$ with $K\subseteq K_n$. If $X$ is locally compact, then hemicompactness is equivalent to $\sigma$-compactness (and Lindel\"ofness).}, the end space can be  intuitively defined in terms of inverse limits (for instance, see \cite{RichardsClassification}).
Instead of directly giving a definition of the end space, we will first introduce the Freudenthal compactification, which can be a bit harder to grasp for general spaces, see, e.g., \cite{DickmanUniform,PorterWoods}. In this paper, we will introduce a definition that works for every Hausdorff locally compact topological space, also formulated in terms of inverse limits; the reader which feels so inclined can check that our definition is equivalent to, e.g., \cite[4X, pp. 336--337]{PorterWoods}.

\begin{Def}
A subset $Y$ of a topological space $X$ will be said to be {\bf bounded} if there exists a compact $K\subseteq X$ with $Y\subseteq K$ (or, equivalently for a Hausdorff space $X$, the closure \( \cl_X(Y) \) of \( Y \) in \( X \) is compact).
Otherwise, $Y$ will be said to be {\bf unbounded}.
\end{Def}

Let $X$ be a Hausdorff locally compact topological space without isolated points. 
We consider finite pairwise-disjoint collections $\mathscr U=\{U_0,\ldots,U_n\}$ of open subsets of $X$ with a distinguished element $U_0$ satisfying:
\begin{itemize}
\item $U_0$ is bounded,
\item $\partial(U_i) = \cl_X(U_i) \ssm U_i $ is compact for all $i\in(n+1)$,
\item $\displaystyle{X=U_0\cup\left(\bigcup_{i=1}^n\cl_X(U_i)\right)}$, and the union is disjoint.
\end{itemize}
We will consistently use the subindex $0$ to refer to the distinguished element of such a collection. Notice that it follows from the third bullet point of the definition that $U_0$ is a regular open set, that is, equal to the interior of its closure in $X$. We denote the set of all such finite collections of open sets, with a distinguished element, as described above, with the symbol $\mathscr K(X)$. Each element $\mathscr U\in\mathscr K(X)$ induces a partition of $X$, denoted:
\begin{equation*}
P_{\mathscr U}=\left\{\{x\}\big|x\in U_0\right\}\cup\left\{\cl_X(U_i)\big|1\leq i\leq n\right\}.
\end{equation*}
We partially order the set $\mathscr K(X)$ by stipulating that, for $\mathscr U,\mathscr V\in\mathscr K(X)$, $\mathscr U\leq\mathscr V$ if and only if the partition $P_{\mathscr V}$ refines the partition $P_{\mathscr U}$. 
Equivalently, if $\mathscr U=\{U_0,\ldots,U_n\},\mathscr V=\{V_0,\ldots,V_m\}$, then $\mathscr U\leq\mathscr V$ if and only if:
\begin{itemize}
\item $U_0\subseteq V_0$,
\item for each $i\in\{1,\ldots,m\}$ there exists a $j\in\{1,\ldots,n\}$ such that $\cl_X(V_i)\subseteq\cl_X(U_j)$.
\end{itemize}
Since a finite union of bounded open sets is a bounded open set, it is not hard to see that $\mathscr K(X)$, equipped with the partial order relation just described, is a directed set. In fact, slightly more is true.

\begin{Lem}
\label{lem:directed}
Let $X$ be a Hausdorff locally compact space without isolated points. For any two $\mathscr U,\mathscr V\in\mathscr K(X)$, there exists $\mathscr W\in\mathscr K(X)$ such that $\mathscr U\leq\mathscr W$, $\mathscr V\leq\mathscr W$, and $\cl_X(U_0)\cup\cl_X(V_0)\subseteq W_0$.
\end{Lem}

\begin{proof}
Cover $\cl_X(U_0)$ and $\cl_X(V_0)$ with finitely many bounded open sets, let the union of those finitely many open sets be $W$, and let $W_0$ be the interior of $\cl_X(W)$, so that $W_0$ is a regular open set, with compact closure, covering $\cl_X(U_0)$ and $\cl_X(V_0)$. Then just define $W_1,\ldots,W_m$ in such a way that the element $\mathscr W=\{W_0,\ldots,W_m\}$ appropriately refines both $\mathscr U$ and $\mathscr V$ (i.e., each $W_j=(U_i\cap V_k)\setminus\cl_X(W_0)$ for some $U_i\in\mathscr U,V_k\in\mathscr V$).
\end{proof}

For each $\mathscr U\in\mathscr K(X)$, we define a topological space $X_{\mathscr U}$ by letting $X_{\mathscr U}$ be the quotient space of $X$ modulo the equivalence relation determined by the partition $P_{\mathscr U}$\footnote{
That is, the underlying space of $X_{\mathscr U}$ is simply $P_{\mathscr U}$, and $ \mathcal V \subseteq P_{\mathscr U}$ is open in $X_{\mathscr U}$ if and only if $\displaystyle{\bigcup_{V\in \mathcal V} V}$ is open in $X$.
}. 
It can be verified easily that the topological space $X_{\mathscr U}$ is compact and Hausdorff. 
Now, given two elements $\mathscr U,\mathscr V\in\mathscr K(X)$, with $\mathscr U\leq\mathscr V$, we define the mapping $\varphi_{\mathscr U,\mathscr V}:X_{\mathscr V}\rightarrow X_{\mathscr U}$ by letting $\varphi_{\mathscr U,\mathscr V}(x)$ be the unique $y\in P_{\mathscr U}$ such that $x\subseteq y$, for every $x\in P_{\mathscr V}$. 
This is well-defined because the partition $P_{\mathscr V}$ refines the partition $P_{\mathscr U}$; it is also easy to verify that this mapping $\varphi_{\mathscr U,\mathscr V}$ is continuous (such mappings---from a finer to a coarser quotient space thus defined---are always continuous).

Thus, we have an inverse system of topological spaces and continuous maps between them, indexed by $\mathscr K(X)$. 
The \emph{Freudenthal compactification} of $X$, denoted $\mathcal F(X)$, is simply defined to be the inverse limit of this directed system. 
Since each of the $X_{\mathscr U}$ is a Hausdorff space, then so will be $\mathcal F(X)$, and since each $X_{\mathscr U}$ is compact, so will be $\mathcal F(X)$ by Tychonoff's theorem. 
Since $X$ is locally compact, for every $x\in X$ and $\mathscr U\in\mathscr K(X)$ there is a terminal segment of elements $\mathscr V\in\mathscr K(X)$, $\mathscr V\geq\mathscr U$, such that $x\in V_0$. Therefore $X$ embeds naturally as a dense open subset of $\mathcal F(X)$, which justifies calling $\mathcal F(X)$ a compactification of $X$. 
Points of the remainder $\mathcal E(X)=\mathcal F(X)\setminus X$ are called {\em ends} of the topological space $X$, and the remainder itself $\mathcal E(X)$ is the {\em space of ends} of $X$. 
An end $e\in\mathcal E(X)$ is formally (since it lives in an inverse limit) an element of the product $\prod_{\mathscr U\in\mathscr K(X)}X_{\mathscr U}$; as a matter of notation, whenever $\mathscr U=\{U_0,\ldots,U_n\}\in\mathscr K(X)$, if $e(\mathscr U)=U_i$ we will write $e\in\hat{U_i}$. 
We record this definition and notations below, for future reference.

\begin{Def}
Given a locally compact Hausdorff topological space without isolated points $X$, we define the {\bf Freudenthal compactification} $\mathcal F(X)$ of $X$ to be the inverse limit of the directed system, indexed by $\mathscr K(X)$, described above. The remainder $\mathcal E(X)=\mathcal F(X)\setminus X$ of this compactification is the {\bf space of ends} of the space $X$; a point in \( \ce(X) \) is an \textbf{end} of \( X \). 
\end{Def}

\begin{Notation}
Let \( X \) be a locally compact Hausdorff topological space without isolated points. Given an open set $U$ with compact boundary, we denote
\begin{equation*}
\hat{U}=\{e\in\mathcal E(X)\big|e(\mathscr U)=U\text{ whenever }U\in\mathscr U\}.
\end{equation*}
For an arbitrary subset $Y\subseteq X$, we will let \( \overline{ Y} \) denote the closure of \( Y \) in \( \cf(X) \); in the case that \( Y \) is open and with compact boundary, we have that $\overline{Y}= \cl_X(Y)\cup\hat{Y}$.
\end{Notation}

When considering a family $\mathscr U=\{U_0,U_1,\ldots,U_n\}\in\mathscr K(X)$, we can always assume that $U_0$ is the only bounded element from $\mathscr U$, as this is the case for a cofinal subset of $\mathscr K(X)$. For if $U_1,\ldots,U_k$ are bounded and $U_{k+1},\ldots,U_n$ are unbounded, we can take (by local compactness) a bounded open set $V\subseteq X$ containing $U_1\cup\cdots\cup U_k$. Without loss of generality, $V$ is a regular open set (if not, replace it with the interior of $\cl_X(V)$). Then, letting $\mathscr V=\{V,U_{k+1}\setminus\cl_X(V),\ldots,U_n\setminus\cl_X(V)\}$, we have that $\mathscr U\leq\mathscr V$ and $V_0$ is the only bounded element of $\mathscr V$. 

Subsets of the form $U\cup\hat{U}$, where $U\subseteq X$ is open with compact boundary, constitute a basis for the topology of $\mathcal F(X)$, whereas subsets of the form $\hat{U}$ where $U\subseteq X$ is an unbounded open set (with compact boundary) form a basis for the topology of $\mathcal E(X)$ (note that, if $U$ is open with compact boundary, then $\hat{U}=\varnothing$ if and only if $U$ is bounded). 
Since $\mathcal E(X)$ is a closed subset of the compact Hausdorff space $\mathcal F(X)$, it is compact and Hausdorff itself. 
Whenever we have $\mathscr U=\{U_0,\ldots,U_n\}\in\mathscr K(X)$, if $U_i$ is unbounded then the set $\hat{U_i}$ constitutes a clopen subset of $\mathcal E(X)$, and therefore $\mathcal E(X)$ has a neighborhood basis of clopen sets, i.e.,~the space of ends is zero-dimensional.

For reference, we record the basic facts about the Freudenthal compactification and the space of ends mentioned above in the following proposition.

\begin{Prop}
\label{prop:ends}
The Freudenthal compactification of a locally compact Hausdorff topological space without isolated points is Hausdorff, without isolated points, and compact; and its space of ends is Hausdorff, compact, and zero-dimensional.
Moreover, if the space is second countable, then so is its Freudenthal compactification and its space of ends. \qed
\end{Prop}

Recall that a \emph{proper map} is a continuous function in which the pre-image of any compact set is compact. A proper map between manifolds can be uniquely extended to a map between their associated Freudenthal compactifications.
We record this here:

\begin{Lem}
Let \( X \) and \( Y \) be locally compact, Hausdorff topological spaces without isolated points.
If \( f \co X \to Y \) is a proper map, then there exists a unique continuous extension \( \bar f \co \cf (X) \to \cf (Y) \) of $f$ that maps ends to ends, in other words, such that the restriction  \(  \hat f=\bar f |_{\ce(X)} \) satisfies \( \hat f \co \ce(X) \to \ce(Y) \).
\end{Lem}

\begin{proof}
This is a standard exercise in the definitions. Simply use the fact that, if $\mathscr V=\{V_0,V_1,\ldots,V_n\}\in\mathscr K(Y)$, then $\mathscr U=\{U_0,U_1,\ldots,U_n\}\in\mathscr K(X)$, where $U_i=f^{-1}[V_i]$.
\end{proof}

As an immediate corollary, we see that a homeomorphism between locally compact Hausdorff  spaces without isolated points induces homeomorphisms between their spaces of ends and their Freudenthal compactifications. 

The following lemma will help simplify our intuition about $\mathscr K(X)$ in the case where $X$ is connected and locally connected (in particular, in the case where $X$ is a manifold, which is the case that we will be concerned with in this paper). Note that Hausdorff connected spaces with more than one point do not have isolated points.

\begin{Lem}
\label{lem:finite}
Let $X$ be a Hausdorff, connected, locally compact, locally connected topological space. Then, every compact subset of $X$ is contained in an open set $U$ such that $\cl_X(U)$ is compact and has finitely many complementary components.
\end{Lem}

\begin{proof}
Let \( K\subseteq X \) be a compact set. Cover \( K \) with finitely many open bounded subsets and let \( V \) denote their union.
Then, cover the compact set \( \cl_X(V) \) with finitely many open bounded sets and let \( W \) denote their union.

Let $\{U_i\big|i\in I\}$ be the collection of connected components of $X\setminus\cl_X(V)$. As $X$ is locally connected, each $U_i$ must be an open set. Hence for each $i\in I$, the set $W_i=W\cup\left(\bigcup_{j\in I\setminus\{i\}}U_j\right)$ is open as well, and so $U_i$ must intersect $W_i$, lest these two sets disconnect the connected space $X$. Since $U_i$ is disjoint from the remaining $U_j$, we conclude that $U_i\cap W\neq\varnothing$, for all $i\in I$.

Now, the family $\{W\}\cup\{U_i\big|i\in I\}$ forms an open cover of the compact set $\cl_X(W)$, and so there are finitely many $i_1,\ldots,i_n\in I$ such that
\begin{equation*}
\cl_X(W)\subseteq W\cup U_{i_1}\cup\cdots\cup U_{i_n}.
\end{equation*}
It follows that $\partial(W)\subseteq U_{i_1}\cup\cdots\cup U_{i_n}$. As the family $\{U_i\big|i\in I\}$ is pairwise disjoint, we must have for each $j\in J=I\setminus\{i_1,\ldots,i_n\}$ that $U_j\cap\partial(W)=\varnothing$ and so it follows that $U_j\subseteq W$ (otherwise $U_j$ would be disconnected by $U_j\cap W$ and $U_j\setminus\cl_X(W)$). 

Now we set $U=V\cup\left(\bigcup_{j\in J}U_j\right)$, which is a bounded open set since it is contained in the compact set $\cl_X(W)$, as argued in the previous paragraph. The set $F=\cl_X(V)\cup\left(\bigcup_{j\in J}U_j\right)$ is closed, as its complement is the open set $U_{i_1}\cup\cdots\cup U_{i_n}$; from here it is easy to see that in fact $F=\cl_X(U)$ and its complementary components are the finitely many sets $U_{i_1},\ldots,U_{i_n}$. Therefore $U$ is as sought.
\end{proof}

The previous lemma essentially says that, in the case where $X$ is connected and locally connected, there are cofinally many elements $\mathscr U=\{U_0,\ldots,U_n\}\in\mathscr K(X)$ such that the $U_i$ for $i\geq 1$ are precisely the connected components of $X\setminus\cl_X(U_0)$. Hence, when dealing with connected locally connected spaces (in particular, when working with manifolds), we may always assume that elements of $\mathscr K(X)$ consist of a compact set $K$ with nonempty interior plus the finitely many components of $X\setminus K$. This observation should make transparent that our definition of the Freudenthal compactification is equivalent to the one that is more commonly used in the context of metrizable manifolds.

\begin{Rem}\label{rem:open-subspace}
We will frequently be interested in considering subspaces of locally compact Hausdorff spaces and analyzing how their Freudenthal compactifications relate to one another. For this, it will be useful to recall that, in a locally compact Hausdorff space $X$, a subspace $Y\subseteq X$ is itself locally compact if and only if it can be written as the intersection of an open subset and a closed subset of $X$~\cite[Theorem 18.4]{Willard}. In particular, if $X$ is a locally compact Hausdorff space and $Y\subseteq X$ is an open subspace, then $Y$ is locally compact as well, and this fact will be used extensively in the remainder of the paper.
\end{Rem}

\begin{Rem}\label{rem:finite}
With essentially the same proof as in Lemma~\ref{lem:finite}, we can prove the following: if $X$ is a Hausdorff, connected, locally compact and locally connected topological space, and $Y\subseteq X$ is an open subspace, then every compact subset $K\subseteq Y$ is contained in an open set $U$ such that $\cl_X(U)\subseteq Y$ and $\cl_X(U)=\cl_Y(U)$ has finitely many complementary components (in $X$). The only modification, with respect to the proof of Lemma~\ref{lem:finite}, that needs to be done is to make sure that every time we cover a compact set with finitely many open bounded sets, we make sure that each of these sets has its closure contained in $Y$ (which is possible by the local compactness of $Y$, see Remark~\ref{rem:open-subspace}).
\end{Rem}

\subsection{The classification of metrizable surfaces.}
As already noted, in two dimensions, the space of ends can be used to give a complete classification of metrizable 2-manifolds up to homeomorphism, which we now describe.

A topological space is \emph{planar} if it can be homeomorphically embedded in the plane \( \br^2 \). 
On a 2-manifold \( S \), an end \( e \) of \( S \) is \emph{planar} (resp. \emph{orientable}) if there exists an open set \( U \subset S \) with compact boundary such that \( e \in \hat U \) and \( U \) is planar (resp. orientable). 
Given a surface \( S \), let \( \ce'(S) \) denote the space of ends that are non-planar and let \( \ce''(S) \) denote the space of ends that are non-orientable.
It follows that \( \ce''(S) \subseteq \ce'(S) \subseteq \ce(S) \). 

If \( S \) is a metrizable 2-manifold, then either
\begin{itemize}
\item \( S \) is orientable, 
\item the complement of every compact subset of \( S \) is a non-orientable 2-manifold, in which case we say \( S \) is \emph{infinitely non-orientable}, or
\item there exists a bounded open subset \( A \) of \( S \) such that \( A \) is non-orientable of finite even (resp. odd) genus and \( S \ssm \cl_S(A) \) is orientable, in which case we say \( S \) has \emph{even (resp. odd) non-orientability type}.
\end{itemize}
We can therefore partition the class of metrizable surfaces into four orientability classes: orientable, infinite non-orientable, even non-orientable, and odd non-orientable. 

\begin{Thm}[The classification of metrizable 2-manifolds, {\cite[Theorems 1 \& 2]{RichardsClassification}}]
\label{thm:classification}
Two  metrizable 2-manifolds \( S \) and \( S' \)  of the same (possibly infinite) genus and orientability class are homeomorphic if and only if \( (\ce(S), \ce'(S), \ce''(S) ) \) and \( ( \ce(S'), \ce'(S'), \ce''(S')) \) are homeomorphic (as triples of spaces).
Moreover, for every triple \( (X,Y, Z) \) of Hausdorff, compact, second countable \footnote{See Remark~\ref{rem:correction} for a discussion of an error in Richards's original statement.}, totally disconnected spaces with \( Z\subseteq Y \subseteq X \) there is a metrizable 2-manifold \( S \) such that \( ( \ce(S), \ce'(S), \ce''(S) ) \) is homeomorphic to \( (X,Y, Z) \) (as triples of spaces). \qed
\end{Thm}

\begin{Rem}
\label{rem:correction}
Theorem~\ref{thm:classification} above is stated in terms of metrizable manifolds; however, Richards's statement \cite[Theorem 1]{RichardsClassification} only includes the assumption that the manifold is separable.
Richards is assuming that a separable manifold is countably triangulable and, as such, his statement is false: the Pr\"ufer manifold (see Appendix~\ref{appendix-prufer}) is an example of a non-metrizable, separable 2-manifold.
However, Rad\'o \cite{RadoUber} proved that every second-countable surface is countably triangulable (see \cite[\S8]{AhlforsRiemann} for a proof).
In addition, Richards's original statement from \cite[Theorem 2]{RichardsClassification} again only claims separability instead of second countability; however, this is an error.  
His argument relies on the false claim that a Hausdorff,  separable, compact, totally disconnected space is homeomorphic to a closed subset of the Cantor set; a counterexample to this is the \v{C}ech--Stone compactification of the natural numbers, which is Hausdorff, compact, separable, and totally disconnected, but fails to be second countable (or even first countable) and hence cannot be realized as a closed subset of the Cantor set. However, with the stronger assumption of second countability replacing separability, the claim is true. As a public service announcement to future readers of Richards's paper \cite{RichardsClassification}, the reader should replace every mention of separability with second countability.
\end{Rem}

\begin{Rem}
The first part of the classification of surfaces \cite[Theorem 1]{RichardsClassification} is originally due to Ker{\'e}kj{\'a}rt{\'o} \cite{Kerekjarto}; however, Richards claims there are gaps in the proof. 
We should also note that, along with Freudenthal, Ker\'ekj\'art\'o independently introduced the notion of an end of a manifold.
\end{Rem}

\subsection{Long and short ends.}

The goal of this paper is to understand surfaces that can be decomposed into a union of a metrizable surface with countably many long planes.
As our understanding of metrizable surfaces---through the classification theorem---relies on their space of ends, we will need to understand the end space of a non-metrizable manifold.
The goal of this section is to introduce the notions of long and short ends; a long end is meant to capture an end that ``requires \( \omega_1 \) time" to escape while a short end is meant to capture the notion of being an end of a metrizable piece of the manifold.

\begin{Def} Let \( X \) be a connected, locally compact, Hausdorff topological space. 
\begin{enumerate}

\item
An end \( e \) of \( X \) is \emph{long} if it is a weak \( P \)-point of \( \cf(X) \), that is, \( e \) is not an accumulation point of any countable subset of \( \cf (X) \).

\item
An end of \( X \) is \emph{short} if it is a \( G_\delta \) point of \( \cf(X) \).

\item
\( X \) has the \emph{end dichotomy property}, or \emph{EDP}, if every end of \( X \) is either long or short. 

\end{enumerate}
\end{Def}

The following well-known lemma (see~\cite[Exercise 3.1.F (a), p. 135]{Engelking}), establishing the equivalence between two possible definitions of a short end, will be used extensively throughout the paper.

\begin{Lem}
If $X$ is a compact Hausdorff space and $x\in X$, then $x$ is a \( G_\delta \) point if and only if it has a countable neighborhood basis. \qed
\end{Lem}

Hence, an equivalent condition for an end to be short is for the end to have a countable neighborhood basis in the Freudenthal compactification.

\begin{figure}[t]
\centering
\includegraphics[scale=0.75]{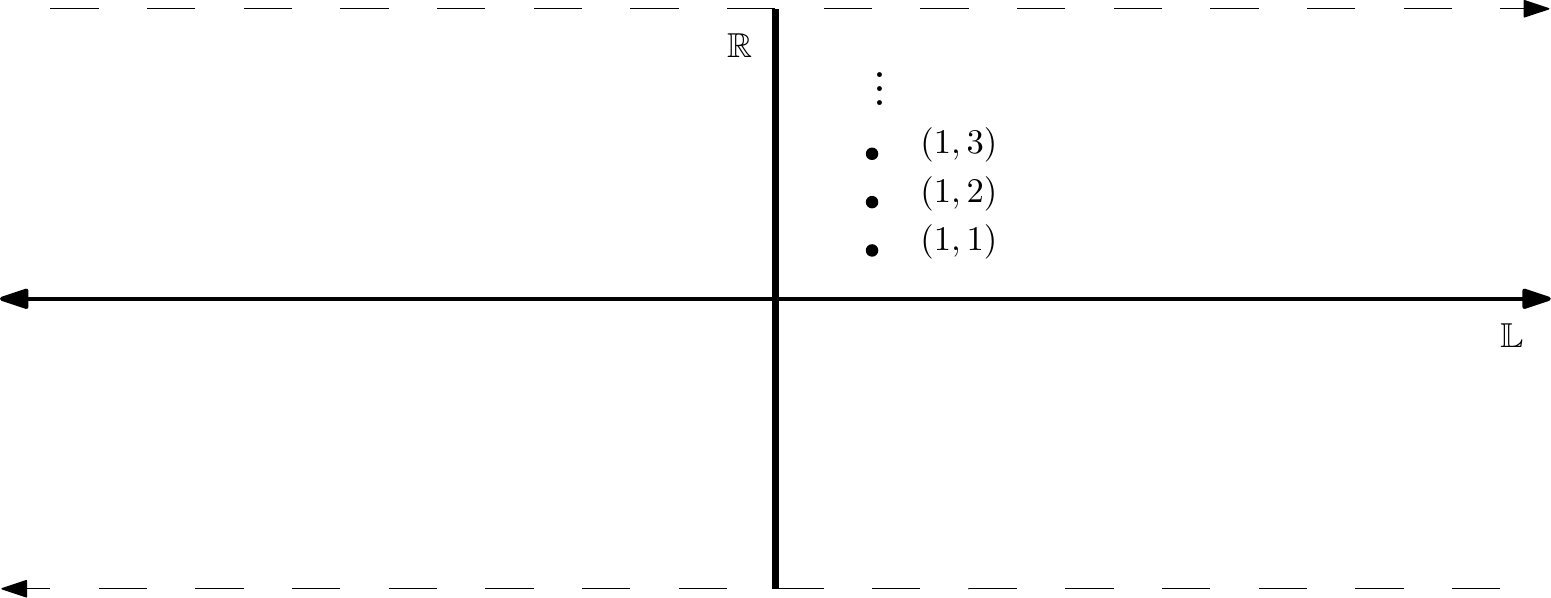}
\caption{The manifold $\mathbb L\times\mathbb R$ is pictured above with the horizontal axis representing \( \mathbb L \) and vertical axis \( \mathbb R \).  Its unique end is neither short nor long (the latter is illustrated by means of a countable sequence converging vertically to such end).}
\label{fig:longtimesr}
\end{figure}

\begin{Ex}
\label{ex:long-line}
The closed long ray is defined to be the linearly ordered topological space $\omega_1\times[0,1)$, ordered lexicographically, and it will be denoted with $\mathbb L^{\geq0}$. The long ray \( \bl^+ \) is the result of removing the minimum element from the linearly ordered topological space $\mathbb L^{\geq0}$. The long line  \( \bl \) is obtained by gluing two copies of the closed long ray along their respective minimum elements.
\begin{enumerate}

\item
\( \bl^+ \) has two ends  corresponding to 0 and \( \omega_1 \); \( 0 \) is short and \( \omega_1 \) is long.

\item
\( \bl \) also has two ends, both of which are long.

\item
\( M = \bl \times \br \) has exactly one end and it is neither short nor long.
Note that the end fails to be long as there is a countable sequence of points in \( M \) that converges to the end in \( \cf (M) \) (such as \( \{ (x,n) \}_{n\in \bn} \), where $x\in\mathbb L$ is fixed and arbitrary).  
Furthermore, the end fails to be short as any single-ended manifold with a short end must be a countable union of compact sets and hence Lindel\"of, which implies metrizable in this context. See Figure~\ref{fig:longtimesr}.
\end{enumerate}
\end{Ex}

\subsection{Relative Freudenthal compactification}

In this subsection, we work with Hausdorff, connected, locally connected and locally compact topological spaces. Recall that, by Lemma~\ref{lem:finite}, for such a space $X$ we can assume without loss of generality that $\mathscr K(X)$ consists of those elements $\mathscr U=\{U_0,\ldots,U_n\}$ such that $U_0$ is a bounded open set, and $U_1,\ldots,U_n$ are the connected components of $X\setminus\cl_X(U_0)$. We now proceed to see that, under appropriate hypotheses, certain subspaces of $X$ provide enough information to approximate the space $\mathcal F(X)$ fairly accurately.

\begin{Def}
Let \(X\) be a Hausdorff, connected, locally connected, locally compact topological space and let \(Y\subseteq X\) be a connected open subspace. We will say that \(Y\) is an {\bf adequate} subspace if, for cofinally many bounded (in $Y$) open sets \(U\subseteq Y\), the (finitely many, by Remark~\ref{rem:finite}) components of \(X\setminus\cl_X(U)\) are all unbounded (in $X$).
\end{Def}

So, let $X$ be a Hausdorff, connected, locally connected, and locally compact space, and let $Y$ be an adequate subspace of $X$. Consider the subfamily $\mathscr K_X(Y)$ of $\mathscr K(X)$ consisting of all $\mathscr U=\{U_0,\ldots,U_n\}\in\mathscr K(X)$ such that $\cl(U_0)\subseteq Y$ and where each $U_i$, for $i\neq 0$, is unbounded. The fact that \(Y\) is an adequate subspace of \(X\) ensures that \(\mathscr K_X(Y)\) is a directed set under the partial order inherited from $\mathscr K(X)$. The directed system of topological spaces and continuous mappings given before on the directed set $\mathscr K(X)$ can be restricted to the directed subset $\mathscr K_X(Y)$; we let \( \cf_X(Y) \) denote the inverse limit of this directed subsystem. Notice that $\mathcal F_X(Y)$ is a compact Hausdorff topological space. 
Since $Y$ itself is locally compact (by Remark~\ref{rem:open-subspace}), for each $y\in Y$ we will have that $y\in U_0$ for cofinally many $\mathscr U=\{U_0,\ldots,U_n\}\in\mathscr K_X(Y)$, and thus $Y$ embeds densely into $\mathcal F_X(Y)$. 
This is to say that \( \cf_X(Y) \) is a compactification of \( Y \); a point of \( \ce_X(Y) = \cf_X(Y) \ssm Y \) is an \emph{end of \( Y \) relative to \( X \)}. 
Notice that $\mathcal F(X)=\mathcal F_X(X)$, and similarly $\mathcal E(X)=\mathcal 
E_X(X)$.

\begin{Lem}
\label{lem:relative-projection}
Let \( X \) be a Hausdorff connected, locally compact, locally connected space, and let $Y\subseteq X$ be an adequate subspace. Then there exists a continuous surjective map \( \pi_Y \co \cf(X) \to \cf_X(Y) \) which is the identity on \( Y\) and such that, if \( x \in X, y \in Y \), and \( \pi_Y(x) = y \), then \( x = y \).
Furthermore, the restriction $p_Y:\mathcal E(X)\rightarrow\mathcal E_X(Y)$ of $\pi_Y$ to $\mathcal E(X)$ is surjective.
\end{Lem}

\begin{proof}
Since $\mathscr K_X(Y)\subseteq\mathscr K(X)$, for each $\mathscr U\in\mathscr K_X(Y)$ there is a natural projection mapping $\pi_{\mathscr U}:\mathcal F(X)\rightarrow X_{\mathscr U}$; thus by the universal property of inverse limits we obtain a continuous mapping $\pi_Y:\mathcal F(X)\rightarrow\mathcal F_X(Y)$; notice that, since each $\pi_{\mathscr U}$ is surjective, so is $\pi_Y$. It is straightforward to check that \(\pi_Y(y)=y\) for \(y\in Y\), and that if \( x \in X, y \in Y, \) and \( \pi_Y(x)=y \), then \( x = y \). Now let $p_Y$ be the restriction of $\pi_Y$ to $\mathcal E(X)$, and let us argue that $\pi_Y$ is surjective. Let \( e\in\mathcal E_X(Y)=\mathcal F_X(Y)\setminus Y\), and pick an \( x\in\mathcal F(X) \) such that \( \pi_Y(x)=e \). We must have \(x\notin Y\); now if we take an arbitrary \(\mathscr U=\{U_0,U_1,\ldots,U_n\}\in\mathscr K_X(Y)\), there will be a unique \(i\in\{1,\ldots,n\}\) such that \(x\in U_i\). Since \(U_i\) is unbounded and with compact boundary, there exists at least one end \(e'\in\mathcal E(X)\) with \(e'\in\hat{U_i}\). It is routine to check that \(p_Y(e')=e\).
\end{proof}

The following notion will be used extensively in the remainder of the paper.

\begin{Def}\label{def:capturing}
Let \(X\) be a Hausdorff connected, locally compact, locally connected space. We will say that a subspace $Y\subseteq X$ {\em captures} an end $e\in\mathcal E(X)$ if for every other end $e'\in\mathcal E(X)\ssm \{e\}$, there exists an unbounded open set $U$ with compact boundary such that $e\in\hat{U}$, $e'\notin\hat{U}$, and $\partial(U)\subseteq\int(Y)$.
\end{Def}

\begin{Lem}
\label{lem:capturing}
Let \( X \) be a Hausdorff connected, locally connected, locally compact space and let $Y\subseteq X$ be an adequate subspace.
If \( Y \) captures every end of $X$, then \( p_Y  \co \ce(X) \to \ce_X(Y) \) is a homeomorphism. 
\end{Lem}

\begin{proof} 
From Lemma~\ref{lem:relative-projection}, we already know that \( p_Y \) is surjective; since $p_Y$ is a continuous mapping between compact Hausdorff spaces, it must be an open map as well. So, we need only show that \( p_Y \) is injective. 

Take two elements $e,e'\in\mathcal E(X)$ with $e\neq e'$. Since $Y$ captures $e$, there is an unbounded open set $U$ with $\partial(U)$ compact, $\partial(U)\subseteq Y$, and $e\in\hat{U}$, $e'\notin\hat{U}$. Take a bounded open set $W_0$, such that $\partial(U)\subseteq W_0\subseteq\cl_X(W_0)\subseteq Y$ and let $\mathscr W=\{W_0,W_1,W_2\}$, where $W_1=U\setminus\cl_X(W_0)$ and $W_2=X\setminus(\cl_X(W_0)\cup\cl_X(U))=X\setminus(\cl_X(W_0)\cup U)$. Since $e\in\hat{W_1}$ and $e'\in\hat{W_2}$, both sets $W_1$ and $W_2$ are unbounded. Hence, $\mathscr W\in\mathscr K_X(Y)$. Since $p_Y(e)(\mathscr W)=W_1\neq W_2=p_Y(e')(\mathscr W)$, this shows that $p_Y$ is injective.
\end{proof}

It remains to show the precise relationship between the relative Freudenthal compactification \(\mathcal F_X(Y)\) and the Freudenthal compactification of \(Y\).

\begin{Lem}\label{lem:surjection-onto-relative}
Let $X$ be a Hausdorff connected, locally compact, locally connected space and let $Y\subseteq X$ be an adequate subspace. Then, there exists a continuous surjective map \(\Pi_Y\co\mathcal F(Y)\rightarrow\mathcal F_X(Y)\) which is the identity on \(Y\), and whose restriction \(P_Y\co\mathcal E(Y)\rightarrow\mathcal E_X(Y)\) to \(\mathcal E(Y)\) is also surjective.
\end{Lem}

\begin{proof}
Take an arbitrary \( \mathscr U=\{U_0,U_1,\ldots,U_n\}\in\mathscr K_X(Y) \) and assume without loss of generality (by Lemma~\ref{lem:finite}) that \(\{U_1,\ldots,U_n\}\) are the connected components of \(X\setminus\cl_X(U_0)\). Note that the set \(\mathscr V=\{U_0,U_1\cap Y,\ldots,U_n\cap Y\}\in\mathscr K(Y)\) and there is an obvious homeomorphism between the spaces \(Y_{\mathscr V}\) and \(X_{\mathscr U}\). Composing this homeomorphism with the projection \(\pi_{\mathscr V}\co\mathcal F(Y)\rightarrow Y_{\mathscr V}\), we obtain onto maps \(\mathcal F(Y)\rightarrow X_{\mathscr U}\) for each \(\mathscr U\in\mathscr K_X(Y)\); the universal property of inverse limits thus yields a continuous onto map \(\Pi_Y\co\mathcal F(Y)\rightarrow\mathcal F_X(Y)\); it is fairly straightforward to show that this map is the identity on \(Y\). One can show, exactly as in Lemma~\ref{lem:relative-projection}, that if $x\in\mathcal F(Y)$ and $\Pi_Y(x)=y\in Y$, then in fact $x=y$; from this it immediately follows that the restriction $P_Y\co\mathcal E(Y)\rightarrow \mathcal E_X(Y)$ of $\Pi_Y$ to $\mathcal E(Y)$ is surjective.
\end{proof}

\begin{Cor}\label{cor:nice-subspace}
Let $X$ be a Hausdorff connected, locally compact, locally connected space and let $Y\subseteq X$ be an open subspace that captures every end of $X$. Then the space $\mathcal E(X)$ is a quotient of $\mathcal E(Y)$.
\end{Cor}

\begin{proof}
By Lemma~\ref{lem:relative-projection}, there exists a continuous surjective mapping $p_Y\co\mathcal E(X)\rightarrow\mathcal E_X(Y)$; furthermore, by Lemma~\ref{lem:capturing}, this map is a homeomorphism because $Y$ captures every end of $X$. Now use Lemma~\ref{lem:surjection-onto-relative} to obtain a continuous surjective map $P_Y\co\mathcal E(Y)\rightarrow\mathcal E_X(Y)$. Then we have a continuous, surjective map $\varphi_Y\co\mathcal E(Y)\rightarrow\mathcal E(X)$ given by composing $P_Y$ with $p_Y^{-1}$. This mapping is open (as it is a continuous map between compact Hausdorff spaces) and therefore it is a quotient map, and so $\mathcal E(X)$ can be seen as a quotient of $\mathcal E(Y)$.
\end{proof}


\section{A criterion for metrizability in terms of ends}
\label{sec:criterion}

Before discussing the relationship between ends and metrizability, let us recall some equivalent conditions for a  manifold to be metrizable, which we will later use without reference.

\begin{Thm}[{\cite[Theorem 2.1]{GauldNonmetrisable}}]
Let \( M \) be a manifold. 
The following are equivalent:
\begin{enumerate}[(i)]
\item \( M \) is metrizable,
\item \( M \) is second countable, 
\item \( M \) is Lindel\"of,
\item \( M \) is hereditarily Lindel\"of, and
\item \( M \) is hemicompact. \qed
\end{enumerate}

\end{Thm}

We now need to formally introduce the definition of a type I manifold.

\begin{Def}
A manifold $M$ is of {\bf type I} if there exists a sequence $\langle M_\alpha\big|\alpha<\omega_1\rangle$ of open Lindel\"of subspaces such that $\cl_M(M_\alpha)\subseteq M_{\alpha+1}$ for all $\alpha<\omega_1$.
\end{Def}

Thus, every metrizable manifold is of type I (just take the constant sequence $M_\alpha=M$). Among non-metrizable manifolds, those that are of type I are much better behaved.  If $M$ is a type I manifold and $\langle M_\alpha\big|\alpha<\omega_1\rangle$ is a sequence witnessing this, it is trivial to modify the sequence so that it furthermore satisfies $M_\alpha=\bigcup_{\xi<\alpha}M_\xi$ whenever $\alpha$ is a limit ordinal; once this condition is satisfied we say that the sequence of $M_\alpha$ is a {\em canonical sequence}. The canonical sequence witnessing that a manifold is of type I is essentially unique, in the sense that any two canonical sequences must agree on a closed unbounded set of $\alpha<\omega_1$.

\subsection{Characterizing metrizability}
By Proposition~\ref{prop:ends}, the Freudenthal compactification of a metrizable manifold is second countable; hence, the end space of a  metrizable manifold is second countable and every end of the manifold is short (as every end has a countable neighborhood basis).
The goal of this section is to prove a partial converse:

\begin{Thm}
\label{thm:characterize}
Let \( M \) be a  manifold in which every end is short.
\begin{enumerate}[(i)]
\item If \( \ce(M) \) is countable, then \( M \) is metrizable.
\item If \( M \) is of type I and \( \ce(M) \) is second countable, then \( M \) is metrizable.
\end{enumerate}
\end{Thm}

In the case of type I manifolds, Theorem~\ref{thm:characterize} gives a complete converse to Proposition~\ref{prop:ends}:

\begin{Cor}
\label{cor:characterize-typeI}
A type I manifold is metrizable if and only if its end space is second countable and every end is short\footnote{Of course if a manifold is compact then its end space is empty, so the statement is only of interest for non-compact manifolds of type I.}. \qed
\end{Cor}

In Appendix~\ref{sectionexamples}, we will describe an example of a (type~I) non-metrizable manifold in which every end is short; so, we see that the second countability is a necessary condition in Corollary~\ref{cor:characterize-typeI}. In an earlier version of this paper, we conjectured that the type I condition was unnecessary; M. Baillif has provided a counterexample to this conjecture---a non-metrizable (not of type I) manifold with a second-countable end space and every end short---which is now described in Appendix~\ref{appendix-prufer}.

Before proving Theorem~\ref{thm:characterize}, we introduce some lemmas culminating with Proposition~\ref{prop:sep-omega}, which is a strengthening of Theorem~\ref{thm:characterize}.  
The main idea is to show that under the conditions on the end space, the manifold is ``almost hemicompact", which is enough in the case of a type I manifold to guarantee metrizability.

\begin{Lem}
\label{lem:connected-cofinal}
Suppose that $M$ is a manifold. Then, the collection of all $\mathscr U=\{U_0,\ldots,U_n\}\in\mathscr K(M)$ satisfying:
\begin{itemize}
\item $U_0$ connected,
\item for all $i\in\{1,\ldots,n\}$, $\partial(U_i)\subseteq\cl_M(U_0)$
\end{itemize}
is cofinal in $\mathscr K(M)$.
\end{Lem}

\begin{proof}
Take $\mathscr U \in \mathscr K(M)$ and assume that $U_0$ either fails to be connected, or $\cl_M(U_0)$ fails to contain some $\partial(U_i)$. Since $\cl_M(U_0)\cup\left(\bigcup_{i=1}^n\partial(U_i)\right)$ is compact (and $M$ is locally compact and locally connected), there are finitely many connected bounded open sets $W_1,\ldots,W_m$ such that $\cl_M(U_0)\cup\left(\bigcup_{i=1}^n\partial(U_i)\right)\subseteq\bigcup_{i=1}^m W_i$. If $\bigcup_{i=1}^m W_i$ is connected, we let $V=\bigcup_{i=1}^m W_i$; otherwise we pick points $x_i\in W_i$ and use that $M$ is path connected to choose $\gamma_i \co [0,1] \rightarrow M$  such that \( \gamma_i(0) = x_i \) and \( \gamma_i(1) = x_{i+1} \) for \( 1 \leq i \leq m-1 \), set \( Y = \left(\bigcup_{i=1}^m \cl_M(W_i) \right) \cup \left( \bigcup_{i=1}^{n-1} \gamma_i \right) \), and find another finite collection of connected bounded open sets $W_{m+1},\ldots,W_k$ such that $Y\subseteq\bigcup_{i=m+1}^k W_i$; in this case, we let $V=\bigcup_{i=m+1}^k W_i$ and note that this set must be connected.

In either case, we have succeeded in obtaining a connected bounded open set $V$ such that $\cl_M(U_0)\cup\left(\bigcup_{i=1}^n\partial(U_i)\right)\subseteq V$. Letting $V_0$ be the interior of $\cl_M(V)$, we obtain a regular open set that is bounded, connected, and still contains $\cl_M(U_0)$ (since $V\subseteq V_0\subseteq\cl_M(V)$ and $V$ is connected); we now just need to let $V_i=U_i\smallsetminus\cl_M(V_0)$, for $1\leq i\leq n$, and define $\mathscr V=\{V_0,\ldots,V_n\}\in\mathscr K(M)$. It is readily checked that $\partial(V_i)\subseteq\cl_M(V_0)$ by construction; furthermore, $\mathscr U\leq\mathscr V$, and so the proof is finished.
\end{proof}

Recall from Definition~\ref{def:capturing} that, if $M$ is a manifold, we will say that a subspace $Y\subseteq M$ {\em captures} an end $e\in\mathcal E(M)$ if for every other end $e'\in\mathcal E(M)\ssm \{e\}$, there exists an unbounded open set $U$ with compact boundary such that $e\in\hat{U}$, $e'\notin\hat{U}$, and $\partial(U)\subseteq\int(Y)$. This notion will be of central importance for the remainder of this section.

\begin{Lem}
\label{lem:end-exhaustion}
Let \( M \) be a manifold.
If \( \ce(M) \) is second countable and every end of \( M \) is short, then there exists a sequence \( \{K_n\}_{n\in\bn} \) of connected compact subsets of \( M \), with $K_n\subseteq\int(K_{n+1})$, such that the open subspace $\bigcup_{n=1}^\infty K_n$ captures every end of $M$. Furthermore, if $\mathcal E(M)$ is countable, then $\bigcup_{n=1}^\infty K_n=M$.
\end{Lem}

\begin{proof}
Let \( \{ V_n \}_{n\in\bn} \) be a collection of unbounded open subsets of \( M \) with compact boundary such that  \( \{\hat{V_n}\}_{n\in\bn} \) is a basis for \( \ce(M) \). Furthermore, in the case where $\mathcal E(M)$ is countable, use the fact that there are only countably many ends, and each of them is short, to choose the $V_n$ in such a way that, for every $e\in\mathcal E(M)$, $\bigcap_{e\in\hat{V_n}}\overline{V_n}=\{e\}$. Note that, in this case, the set $Y=M\setminus\left(\bigcup_{n=1}^\infty V_n\right)$ must be bounded (or else, there would be an end $e\in\overline{Y}$, contradicting our choice of the $V_n$).

Now, whether $\mathcal E(M)$ is countable, or just second countable, in either case for each $n\in\mathbb N$ there exists an element $\mathscr U_n=\{U_0^n,\ldots,U_{k_n}^n\}\in\mathscr K(M)$ such that $U_1^n=V_n$; furthermore, by Lemma~\ref{lem:connected-cofinal}, we may assume without loss of generality that $\partial(U_1^n)\subseteq\cl_M(U_0^n)$ for all $n\in\mathbb N$. Start by defining $W_0^0=\varnothing$ and then, using Lemmas~\ref{lem:directed} and~\ref{lem:connected-cofinal}, we can inductively build $\mathscr W_n=\{W_0^n,\ldots,W_{m_n}^n\}\in\mathscr K(M)$ such that $\mathscr W_{n-1},\mathscr U_n\leq\mathscr W_n$ and such that \( W_0^n \) is connected. We then have that $\partial(U_1^n)\subseteq\cl_M(U_0^n)\cup\cl_M(W_0^{n-1})\subseteq W_0^n$; furthermore, in the case where $\mathcal E(M)$ is countable, we can ensure that $Y\subseteq W_0^1$. Note that the collection $\{\hat{W_i^n}\, | \, n \in \bn, 1\leq i \leq m_n\}$ forms a (countable) basis for $\mathcal E(M)$: for if $e\in E(M)$ and $e\in\hat{V_n}$, then there is some $i\leq m_n$ with $e\in\hat{W_i^n}\subseteq\hat{V_n}$.

Now we simply let $K_n=\cl_M(W_0^n)$. By construction $K_n\subseteq \int(K_{n+1})$, and this easily implies that $U=\bigcup_{n=1}^\infty K_n$ is an open subspace of $M$. Moreover, if $e,e'\in\mathcal E(M)$ with $e\neq e'$, then (as $\mathcal E(M)$ is Hausdorff and hence $T_1$) there are $n,i$ such that $e\in\hat{W_i^n}$ and $e'\notin\hat{W_i^n}$. Since $\partial(W_i^n)\subseteq\cl_M(W_0^n)=K_n\subseteq K_{n+1}\subseteq\int(U)$, it follows that $U$ captures the end $e$; since $e$ was arbitrary, this finishes the proof that $U$ captures every end of $M$. Furthermore, in the case where $\mathcal E(M)$ is countable, we actually have $\bigcup_{n=1}^\infty K_n\supseteq Y\cup\bigcup_{n=1}^\infty V_n=M$.
\end{proof}

\begin{Prop}
\label{prop:sep-omega}
Let \( M \) be a manifold.
If \( \ce(M) \) is second countable and every end of \( M \) is short, then either 
\begin{enumerate}[(i)]
\item \( M \) is separable, or
\item there exists an open Lindel\"of subset \( U \) of \( M \) such that \( M \ssm \cl_M(U) \) has uncountably many components (all of which are open) and \( U \) captures every end of \( M \).

\end{enumerate}
Moreover, if either \( \ce(M) \) is countable or \( M \) is a type I manifold, then \( M \) is Lindel\"of (and hence, metrizable).
\end{Prop}

\begin{proof}
Let \( M \) be as in the proposition. If $M$ is compact  then it is separable, so assume that $M$ is non-compact. Let \( \{K_n\}_{n\in\bn} \) be the sequence of compact sets given by Lemma~\ref{lem:end-exhaustion} and let \( U = \bigcup_{n\in\bn} K_n \)---note that \( U \) is open, connected,  Lindel\"of, and captures every end of \( M \). Furthermore, if $\mathcal E(M)$ is countable, then $M=\bigcup_{n=1}^\infty K_n$ and so $M$ is Lindel\"of, which establishes the first half of the ``moreover'' part of the proposition.

In the general case, there are four possibilities:

(1) \( U = M \) implying \( M \) is Lindel\"of (and hence separable)

(2) \( \cl_M(U) = M \) implying  \( M \) is separable, 

(3) \( M \ssm \cl_M(U) \) has countably many components, or

(4) \( U \) is as in (ii) (note that the complementary components of a closed subset of a locally connected space are open).

As already noted, \( M \) is separable in the first two cases.
We claim that \( M \) is also separable in case (3):
assume \( M \ssm \cl_M(U) \) has countably many components and observe that the closure of each component of \( M \ssm \cl_M(U) \) in \( \cf(M) \) contains at most one end.
To see this, let \( Q \) be such a component and suppose that \( e \) and \( e' \) are ends in \( \overline{Q} \). Since \( U \) captures every end of $M$, there is an unbounded open set \( V \) with compact boundary $\partial(V)\subseteq U$ such that $e\in\hat{V}$ and $e'\notin\hat{V}$. 
Since \( \overline Q \) contains both \( e \) and \( e' \), it must be that \( Q \) intersects both \( V \) and \( M \ssm V \) and hence \( Q \) intersects \( \partial(V) \subset U \) (otherwise \( Q \) would be disconnected), which is impossible since \( Q \subset M \ssm \cl_M(U) \).

Let \( \{B_k\}_{k\in K} \) denote the bounded components of \( M \ssm \cl_M(U) \) and let \( \{Q_j\}_{j\in J} \) denote the unbounded components, where \( K \) and  \( J \) are both countable indexing sets.
By the above discussion, for each \( j \in J \), there exists \( e_j \in \ce(M) \) such that \( \{e_j\} = \overline{Q_j}\cap\mathcal E(M) \).
For each \( j \in J \), choose a sequence of unbounded open sets \( \{V_n^j\}_{n\in\bn} \) such that \( \partial (V_n^j) \) is compact and \( \{ e_j \} = \bigcap_{n\in\bn} \overline{V_n^j} \) (this can be done because \( e \) is short).
Since \( V_n^j \) is open and \( e \in \hat V_n^j \) for each \( n \in \bn \), we must have that \( Q_j \ssm V_n^j \) is bounded.
Note that every bounded set in a manifold is separable\footnote{Every bounded subset of a manifold can be covered with a bounded open subset; any bounded open subset of a manifold is itself a Lindel\"of manifold and hence it is metrizable and second countable. This means that every bounded subset of a manifold is second countable, and therefore also separable.} and so each of the $B_k$, as well as the $Q_j\setminus V_n^j$, are separable. Therefore, we can write \( M \) as a countable union of separable sets,
\[
M = \cl_M(U) \cup \left(\bigcup_{k\in K} B_k \right) \cup \left(\bigcup_{j\in J}\bigcup_{n\in\bn} (Q_j\ssm V_n^j) \right),
\]
and hence \( M \) is separable. This finishes the proof of the main statement of the theorem.

Now, for the second half of the ``moreover'' part, let us assume that \( M \) is of type I and let \( \{M_\al : \al < \omega_1\} \) be a canonical sequence for \( M \).
If \( M \) falls into cases (1)--(3), then it is separable; in this case, it is not difficult to see that this implies there exists an ordinal \( \be < \omega_1 \) such that \( M_\al = M \) for all \( \al \geq \be \).
Hence, \( M \) is Lindel\"of.
We finish by noting that case (4) is impossible:
As \( U \) is Lindel\"of, there exists an ordinal \( \al < \omega_1 \) such that \( \cl_M(U) \subset M_\al \).  
Again as \( M \) is locally connected, the components of \( M \ssm \cl_M(U) \) are open and, in addition, as \( M \) is connected, each component of \( M \ssm \cl_M(U) \) intersects \( M_\al \).
Therefore, the components of \( M \ssm \cl_M(U) \) together with \( U \) give an open cover of \( M_\al \).
It now follows that \( M\ssm \cl_M(U) \) has countably many components as \( M_\al \) is Lindel\"of.
\end{proof}

Theorem~\ref{thm:characterize} is now just a special case of Proposition~\ref{prop:sep-omega}.
To see this, recall---as noted in the beginning of the section---that a manifold is Lindel\"of if and only if it is metrizable.


\section{The general bagpipe theorem}
\label{sec:bagpipe}
In this section, we prove our main theorems.
In each subsection, we recall the theorem of Nyikos we aim to generalize and restate several of his results in terms of the space of ends and the language introduced thus far.

\subsection{The Bagpipe Lemma}
The first theorem of Nyikos we consider holds for manifolds of any dimension.
Before stating Nyikos's result, we recall two definitions from~\cite{NyikosTheory}:
a locally connected space \( M \) is \emph{trunklike} if, given any closed Lindel\"of subset \( C \), \( M\ssm C \) has at most one non-Lindel\"of component.
A topological space is \emph{\( \omega \)-bounded} if every countable subset has compact closure.

Nyikos proves that an \( \omega \)-bounded manifold is the union of a compact metrizable space and finitely many trunklike manifolds:

\begin{Thm}[``The Bagpipe Lemma" {\cite[Theorem 5.9]{NyikosTheory}}]
Every \( \omega \)-bounded manifold \( M \) has an open Lindel\"of subset U such that \( M \ssm \cl_M(U) \) is the disjoint union of finitely many trunklike manifolds. \qed
\end{Thm}

It is worth noting that a trunklike manifold need not have a long end: for example, the space \( \bl^+\times \br \) is trunklike without a long end.

Our first result regarding the analysis of non-metrizable manifolds in terms of their ends is yet another characterization of $\omega$-bounded manifolds, so there is a sense in which the following directly extends Nyikos's original result.

\begin{Thm}
\label{lem:finite-long}
A  manifold $M$ is \( \omega \)-bounded if and only if each end of $M$ is long. In this case, moreover, the end space $\mathcal E(M)$ is finite.
\end{Thm}

\begin{proof}
First assume that \( M \) is \( \omega \)-bounded. If \( M \) is compact, then it has no ends and hence it is vacuously true that each end is long; hence, we assume that \( M \) is non-compact.
We first argue that no end of \( M \) is short: if \( e \) were a short end of \( M \), then we could choose a countable neighborhood basis \( \{W_n\}_{n\in\bn} \) of $e$ in \( \cf(M) \).
Then, pick a point \( x_n \in W_n \cap M \) so that \( A = \{x_n : n \in \bn\} \) is a countable subset of \( M \) whose closure in \( M \) fails to be compact (as it misses $e$), which contradicts \( M \) being \( \omega \)-bounded.

Now, let \( U \) be the Lindel\"of subset given by the bagpipe lemma.
As \( M \) is \( \omega \)-bounded and \( U \) is separable, we must have that \( \cl_M(U) \) is compact. 
If \( T_1, \ldots, T_n \) are the components of \( M \ssm \cl_M(U) \), then \( T_i \) is trunklike for each \( i \in \{1, \ldots, n \} \).
Note that \( \hat T_i \) is open in \( \ce(M) \) and  \( \ce(M) = \bigsqcup_{i=1}^n \hat T_i \).

We claim that \( \ce(M) \) has cardinality \( n \).
First observe that, since each \( T_i \) is unbounded, \( \hat T_i \) is nonempty.
This implies that \( \ce(M) \) has at least \( n \) points. 
Now suppose that some \( \hat T_i \) contains at least two ends $e,e'$, $e\neq e'$.
Since \( T_i \cup \hat T_i \) is an open neighborhood of both \( e \) and \( e' \) in \( \cf(M) \), we can find an element $\mathscr V=\{V_0,V_1,V_2,\ldots,V_m\}\in\mathscr K(X)$ such that $V_1,V_2\subseteq T_i$ are disjoint and $e\in\hat{V_1}\setminus\hat{V_2}$, $e'\in\hat{V_2}\setminus\hat{V_1}$. But then \( \cl_M(V_0)\cap T_i \) is a closed Lindel\"off subset of \( T_i \) and \( T_i \ssm \cl_M(V_0) \) has (at least) two non-Lindel\"of components, $V_1$ and $V_2$.
This contradicts \( T_i \) being trunklike; hence, \( \ce(M) \) has cardinality \( n \).

We now have that \( \ce(M) \) is finite (hence discrete). Since we are assuming that $M$ is $\omega$-bounded, the closure of any countable subset in \( M \) is compact. If $e\in\mathcal E(M)$ was not long, there would be a countable set $A\subseteq\mathcal F(M)$ with $e\in\cl_{\mathcal F(M)}(A)\setminus A$. 
But $\mathcal E(M)$ is finite, thus $Y=A\cap M=A\setminus\mathcal E(M)$ is a cofinite subset of $A$ and therefore $Y$ and $A$ have the same accumulation points in $\mathcal F(M)$. Hence we must have $e\in\overline{Y}$, however $Y$ is countable and hence its closure in $M$ is compact, which means that $\overline{Y}\subseteq M$ and this is a contradiction. This means that each of the finitely many ends of $M$ is long. 

For the converse, if each end of $M$ is long, then the closure of every countable subset of \( M \) in \( \cf(M) \) is disjoint from the ends.
Hence (since $\cl_M(A)=\cl_{\mathcal F(M)}(A)$), the closure of \( A \) in \( M \) is compact whenever $A$ is countable; therefore \( M \) is \( \omega \)-bounded (and {\it a fortiori}, by the forward direction, $\mathcal E(M)$ is finite). 
\end{proof}

An end of a manifold is \emph{isolated} if it is an isolated point of the space of ends.
As every end of an \( \omega \)-bounded manifold is isolated and long by Theorem~\ref{lem:finite-long}, we are led to the following definition:

\begin{Def}
\label{def:trunk}
A \emph{long trunk} is a type I trunklike manifold with the EDP and finitely many ends, exactly one of which is long.
\end{Def}

Combining Theorem~\ref{lem:finite-long} and Definition~\ref{def:trunk}, we can give a strengthened version of the Bagpipe Lemma as follows:

\begin{Thm}[``The Bagpipe Lemma: Promoting Trunk-Like to Long Trunks"]
\label{thm:bag-lemma}
Every \( \omega \)-bounded manifold \( M \) has an open Lindel\"of subset U such that \( X \ssm \cl_M(U) \) is the disjoint union of finitely many long trunks.
\end{Thm}

The main theorem of this subsection is a generalized version of Theorem~\ref{thm:bag-lemma} and its converse; in order to state it, we first recall several standard definitions.

\begin{Def}\hfill
\begin{enumerate}
\item A {\bf bordered \( n \)-manifold} is a  connected Hausdorff topological space in which every point has an open neighborhood homeomorphic to an open subset of the closed upper half space \( \bar{\mathbb{H}}^n = \{ (x_1, \ldots, x_n) \in \br^n : x_n \geq 0 \} \).
\item If $M$ is a bordered $n$-manifold and $x\in M$, we say that
\begin{enumerate}
\item $x$ is a {\bf manifold point} of $M$ if it has an open neighborhood homeomorphic to $\mathbb R^n$, and
\item $x$ is a {\bf boundary point} of $M$ if it has an open neighborhood homeomorphic to $\bar{\mathbb{H}}^n$.
\end{enumerate}
\item A subset \( B \) of a (bordered) manifold \( M \) is {\bf collared} if there exists an embedding \( \iota\co B\times [0,1) \to M \) such that \( \iota(b,0) = b \) for every \( b \in B \); \( B \) is {\bf bi-collared} if there exists an embedding \( h \co B \times (-1,1) \to M \) such that \( h(b,0) = b \) for every \( b \in B \). 

\end{enumerate}
\end{Def}

Note that every $n$-manifold is also a bordered $n$-manifold (one with an empty set of boundary points), and that the set of manifold points of a bordered $n$-manifold is an $n$-manifold.
We are now ready to state the main theorem of this subsection.

\begin{Thm}[The General Bagpipe Lemma]
\label{thm:gen-bag-lemma}
Let \( M \) be a manifold. 
\begin{enumerate}
\item
If \( M \) satisfies the EDP and either (i) \( M \) is of type I and has second countable end space or (ii) \( M \) has countable end space, then there exists an open  Lindel\"of subset \( U \) of \( M \) such that \( \cl_M(U) \) is a bordered manifold with compact boundary components, each of which is bi-collared, and \( M \ssm \cl_M(U) \) is the disjoint union of countably many long trunks.
\item
If there exists an open  Lindel\"of subset \( U \) of \( M \) such that \( M \ssm \cl_M(U) \) is the disjoint union of countably many long trunks, then \( M \) satisfies the EDP, has second countable end space, and is of type I.
\end{enumerate}
\end{Thm}

Note that the original bagpipe lemma does not assume the manifold to be of type I, but every \( \omega \)-bounded manifold is of type I. It follows immediately from Theorem~\ref{thm:gen-bag-lemma} that every manifold with a countable end space and the EDP must be of type I.

Before getting to the proof, we need three preliminary results.

\begin{Lem}
\label{lem:isolated}
If the end space of a manifold is second countable, then every long end is isolated.  
\end{Lem}

\begin{proof}
Under the assumptions, the end space of the manifold $M$ is first countable; hence, every $e\in\mathcal E(M)$ is either isolated, or it belongs to the closure of a countable set in $\mathcal E(M)\subseteq \cf(M)$. Therefore, by the definition of a long end, we conclude that every long end $e\in\mathcal E(M)$ must be isolated.
\end{proof}

We now record a useful corollary stemming from significant results in manifold theory establishing 
the fact that every non-compact second-countable manifold admits a handlebody decomposition\footnote{In fact, this holds for every second-countable manifold that is not an unsmoothable 4-manifold.} (this has a long history, but we point the reader to \cite[Theorem 9.2 and Theorem 8.2]{FreedmanTopology}).
Without concerning ourselves with the definition of a handlebody, we record a straightforward corollary, Proposition~\ref{prop:submanifold} below.

\begin{Prop}
\label{prop:submanifold}
Let \( M \) be a non-compact manifold.
If \( K \) is a compact subset of \( M \), then there exists a compact bordered manifold \( N \)  contained in \( M \) such that \( K \) is contained in the interior of \( N \) and \( \partial (N) \) is bi-collared in \( M \). 
\end{Prop}

\begin{proof}
We first show there exists a compact bordered submanifold \( N' \) of \( M \) containing in \( K \) in its interior.
For a short proof when \( M \) is second countable, we refer the reader to \cite[Proposition 3.17]{GauldNonmetrisable}.
For a general manifold, simply observe that every compact subset is contained in an open connected Lindel\"of subset; we can then apply the second-countable case to this subset.
Now, the boundary of a bordered manifold is collared \cite[Theorem~2]{BrownLocally}; hence, \( \partial (N') \) is collared in \( N' \) (and hence in \( M \)).
Moreover, since \( K \) is closed and disjoint from \( \partial (N') \), we can choose an embedding \( \iota \co \partial (N') \times [0,1) \to N' \) defining a collar of \( \partial (N') \) such that the collar is disjoint from \( K \).
Now, we conclude the proof by setting \( N \) to be the complement of \( \iota( N' \times [0, 1/2)) \) in \( N' \).
\end{proof}

Note that since \( \partial (N) \) is bi-collared, the closure of the complement of \( N \) is also a submanifold.
This fact will implicitly be used repeatedly in the following arguments. 

\begin{Lem}
\label{lem:connected}
If \( e \) is an isolated end of a manifold \( M \), there is a connected open subset \( U \) of \( M \) such that \( \hat U =  \{e\} \),  \( \cl_M(U) \) is a bordered manifold with bi-collared boundary, and \( \partial (U) \) is compact. 
\end{Lem}

\begin{proof}
As \( e \) is isolated, there is an element $\mathscr U=\{U_0,\ldots,U_n\}\in\mathscr K(X)$ such that $\hat U_1 = \{e\}$. 
We may assume without loss of generality that $U_1$ is connected (if it is not, split it into the connected component containing $e$ and the rest to get a larger element of $\mathscr K(X)$). 
Since \( \partial (U_1) \) is compact,  Proposition~\ref{prop:submanifold} yields a compact bordered submanifold \( N \) of \( M \) with bi-collared boundary containing \( \partial(U_1) \) in its interior. 
To finish, let \( U \) be the unbounded component of \( U_1 \cap (M \ssm N) \). 
\end{proof}

\begin{proof}[Proof of Theorem~\ref{thm:gen-bag-lemma}]
We first prove (1):
Let $M$ be a manifold with the EDP.
Additionally, we assume that either \( M \) is a manifold of type I with second countable end space or \( M \) is an arbitrary manifold with countable end space.
If \( M \) has no long ends, then---by Proposition~\ref{prop:sep-omega}---\( M \) is Lindel\"of and we set \( U = M \); so, we will also assume that \( M \) has at least one long end. 
By Lemma~\ref{lem:isolated}, each long end of \( M \) is isolated; hence, by Lemma~\ref{lem:connected}, for each long end \( \ell \) of \( M \) we can choose an unbounded open set $U_\ell$ such that $\hat{U_\ell}=\{\ell\}$ and $\cl_M(U_\ell)$ is a compact bordered manifold with bi-collared boundary. Moreover, since $\mathcal E(M)$ is second countable, it has only countably many isolated points and thus there are only countably many such $\ell$. This allows us to inductively choose the $U_\ell$ in such a way that for any two long ends \( \ell_1 \) and \( \ell_2 \), the intersection of \( U_{\ell_1} \) and \( U_{\ell_2} \) is empty.

We claim that \( U_\ell \) is a long trunk.
Note that \( \cl_M(U_\ell) \) is Hausdorff and locally compact; hence, \( \cl_M(U_\ell) \) has a well-defined Freudenthal compactification.
Moreover, by construction, it is not difficult to see that \( \cl_M(U_\ell) \) has a single end, which is long and corresponds to \( \ell \).
Arguing as in Theorem~\ref{lem:finite-long}, we see that \( \cl_M(U_\ell) \) is \( \omega \)-bounded; then, applying \cite[Theorem 4.10]{GauldNonmetrisable}---or rather a slight modification to bordered manifolds---\( \cl_M(U_\ell) \) is of type I.
By intersecting a canonical sequence for \( \cl_M(U_\ell) \) with \( U_\ell \), we see that \( U_\ell \) is of type I as well.
Now if \( \cl_M(U_\ell) \) has \( n \) boundary components, then it follows that \( U_\ell \) is \( (n+1) \)-ended, with one end corresponding to \( \ell \) and the others (located where the components of $\partial(U_\ell)$ should be) short; hence, \( U_\ell \) is a long trunk.

Let $\ell_k$ ($k\in I$, with $I$ countable) enumerate all the long ends. For each $k\in I$, let \( \partial_1^k, \ldots, \partial_{m_k}^k \) denote the components of \( \partial(U_{\ell_k}) \) and choose paths \( \gamma_i^k \) in \( U_{\ell_k} \) connecting \( \partial_i^k \) and \( \partial_{i+1}^k \) for \( 1 \leq i < m_k \). 

Let \( \Gamma_1=\partial(U_{\ell_1}) \cup \gamma_1^1 \cup \cdots \cup \gamma_{m_1-1}^1 \) and apply Proposition~\ref{prop:submanifold} to obtain an open set \( V'_{\ell_1} \) of \( M \) such that \( \cl_M(V'_{\ell_1}) \) is a compact bordered submanifold of \( M \) with bi-collared boundary and \( \Gamma_1 \subseteq V'_{\ell_1} \).
Moreover, we can assume that \( V'_{\ell_1} \) is connected by taking the connected component of \( V'_{\ell_1} \) containing the connected set \( \Gamma_1\).
Note that \( M \ssm V'_{\ell_1} \) has finitely many connected components (since \( \partial(V'_{\ell_1}) \) has finitely many components); hence, by replacing \( V'_{\ell_1} \) with the union of \( V'_{\ell_1} \) with its bounded complementary components, we may also assume that each component of \( M \ssm V'_{\ell_1} \) is unbounded.

Now, proceeding recursively, assuming we have already defined \( \Gamma_{k} \) and \( V'_{\ell_k} \), let $\gamma_k$ be a path connecting $\Gamma_k$ with \( \partial(U_{\ell_{k+1}}) \), and let \( \Gamma_{k+1}= V'_{\ell_k}\cup\gamma_k \cup \partial(U_{\ell_{k+1}}) \cup \gamma_1^{k+1} \cup \cdots \cup \gamma_{m_{k+1}-1}^{k+1} \). 
As in the base case, apply Proposition~\ref{prop:submanifold} to obtain an open set \( V'_{\ell_{k+1}} \) of \( M \) such that \( \cl_M(V'_{\ell_{k+1}}) \) is a compact bordered submanifold of \( M \) with bi-collared boundary and \( \Gamma_{k+1} \subset V'_{\ell_{k+1}} \); again, we may assume that \( V_{\ell_{k+1}} \) is connected and that each of its complementary components is unbounded. 

Notice that there is a distinguished component \( Z_{\ell_k} \)  of \( M \ssm V'_{\ell_k} \) yielding a neighborhood of \( \ell_k \) and, in particular, \( \partial (Z_{\ell_k}) \subset U_{\ell_k} \) (this follows from the fact that the complement of  \( \partial (V'_{\ell_k}) \) is disconnected and that \( \partial (V'_{\ell_k}) \) is disjoint from \( \Gamma_k \)).
Now, for \( k \in I \),  define \( V_{\ell_k} = V'_{\ell_k}\ssm \left( \cl_M(Z_{\ell_1}) \cup \cdots \cup \cl_M(Z_{\ell_k})\right) \) and observe that \( V_{\ell_k} \) is connected, \( \cl_M(V_{\ell_k}) \) is a bordered submanifold of \( M \) with bi-collared boundary, and \( \partial (Z_{\ell_j}) \) is a boundary component of \( \cl_M(V_{\ell_k}) \) for each \( j \leq k \). 
Finally, let \( V = \bigcup_{k\in I} V_{\ell_k} \) and let \( Z = \bigcup_{k \in I} \partial(Z_{\ell_k}) \).
It follows that \( V \) is an open connected Lindel\"of subset of \( M \), \( Z \) is bi-collared in \( M \), and \( V \cup Z \) is a bordered submanifold whose set of boundary points (as a bordered manifold) is \( Z \). 

Let \( U = \left(M \ssm \bigcup_{\ell \in \mathcal{L}(M)} U_\ell\right) \cup V \).
Then, \( U \) is an open, connected subset of \( M \) such that \( \cl_M(U) = U \cup Z \) is a bordered manifold with bi-collared boundary in \( M \) and compact boundary components (observe that the complement in $M$ of $U\cup Z$ is open, as it is the union of the $Z_\ell$). Moreover, each component of \( M \ssm U \) is unbounded (such components are the $\cl_M(Z_\ell)$). Furthermore, one can argue that $U$ is an adequate subspace of $M$. For, whenever $F\subseteq U$ is a closed set and $B$ is a bounded (in $M$) connected component of $M\setminus F$, we consider the set $B'=\cl_M(B)\setminus F$ and note that $B\subseteq B'\subseteq\cl_M(B)$, so $B'$ must be connected. If $B'$ were to intersect $M\setminus U$, then $B'$ would need to intersect one of the unbounded components of $M\setminus U$, contradicting our assumption about $B'$. Hence $B'$ and so also $\cl_M(B)$, is a subset of $U$ and so $B$ is bounded in $U$. Thus, whenever $W$ is a bounded open set with $\cl_M(W)\subseteq U$, every bounded (in $M$) component of $M\setminus\cl_M(W)$ is also bounded in $U$, and therefore we can always enlarge $W$ so that every component of $M\setminus\cl_M(W)$ is unbounded in $M$, showing that $U$ is an adequate subspace of $M$. With the properties from the previous paragraph, Lemma~\ref{lem:relative-projection} yields continuous surjections \( \pi_U \co \cf(M) \to \cf_M(U) \) and \( p_U \co \ce(M) \to \ce_M(U) \).

It remains to prove that \( U \) is Lindel\"of. First observe that, by construction, \( U \) is an open, connected subset of \( M \). Furthermore, let us argue that \( U \) captures every end of \( M \): 
by construction, \( U \) captures every long end of \( M \).  Now, suppose \( e \) is a short end and \( e' \) is any other end of \( M \).
Let \( W \) be an unbounded open set with compact boundary such that \( e \in \hat W \) and \( e' \notin \hat W \). 
By the compactness of \( \partial(W) \), there are at most finitely many long ends \( \ell \) such that \( \partial(W) \cap U_{\ell} \neq \varnothing \), label them \( \ell_1, \ldots, \ell_k \). 
Replacing \( W \) by
\[ W' = W \setminus \left(\bigcup_{i=1}^k \cl_M(U_{\ell_i}) \right), \]
we have that \( \partial(W') \subset U \) with \( e \in \hat W' \) and \( e' \notin \hat W' \).
Now, we have shown that \( U \) captures every end of \( M \).

Applying Lemma~\ref{lem:capturing}, we have that \( p_U \co \ce(M) \to \ce_M(U) \) is a homeomorphism.
Let \( P_U \co \ce(U) \to \ce_M(U) \) be the map given by Lemma~\ref{lem:surjection-onto-relative}, then, from Corollary~\ref{cor:nice-subspace}, we have that \( \vp_U = p_U^{-1} \circ P_U \co \ce(U) \to \ce(M) \) is a quotient map. 
By the construction of \( U \)---and, in particular, the fact \( Z_\ell \) has finitely many boundary components---the map \( \vp_U \) is finite-to-one.
Moreover, if the pre-image of an end of \( M \) under the map \( \vp_U \) has more than one point, then the end must be long and hence isolated.
In particular, since \( \ce(U) \) is Hausdorff, any end of \( U \) that is contained in the pre-image of a long end of \( M \) must be isolated in \( \ce(U) \).
Let \( \{e_j\}_{j\in J} \) be some indexing of the ends of \( U \) whose image under \( \vp_U \) is a long end of \( M \) by a countable set \( J \); additionally, let \( \{A_n\}_{n\in\bn} \) be a countable basis for \( \ce(M) \).
Then, \( \{ \{e_j\} \}_{j\in J} \cup \{ \vp_U^{-1}[A_n] \}_{n\in\bn} \) is a countable basis for \( \ce(U) \); hence, \( \ce(U) \) is second countable.

Let us next argue that every end of \( U \) is short.
First, let us consider the case where \( e \in \ce(U) \) and  \( \vp_U(e) \) is short. 
In this case, the map \( \pi_U \co \cf(M) \to \cf_M(U) \) From Lemma~\ref{lem:relative-projection} is open and hence \( P_U(e) \) is a \( G_\delta \) point of \( \cf_M(U) \) since \( \vp_U(e) \) is a \( G_\delta \) point of \( \cf(M) \).
In particular, \( \{e\} = \Pi_U^{-1}[\{P_U(e)\}] \) and hence, by the continuity of \( \Pi_U \), \( e \) is also a \( G_\delta \) point.

Now, let us consider the case when \( e \in \ce(U) \) and \( \ell = \vp_U(e) \) is long. 
Fix a net \( \{x_\lambda\}_{\lambda\in \Lambda} \) in \( U \) converging to \( e \) in \( \cf(U) \).  
Then there exists a unique boundary component, call it \( \partial_e \), of the bordered manifold \( \cl_M(U) \)---and of \( V_\ell \)---containing the accumulation points of \( \{x_\lambda\}_{\lambda\in\Lambda} \) in \( M \): this follows from the fact that the boundary components of \( \cl_M(U) \) are compact subsets of \( M \) and disjoint compact subsets can always be separated by open sets in a Hausdorff space. 
Since \( \partial_e \) is compact, we can choose open subsets \( \{A_n\}_{n\in\bn} \) of \( M \) such that \( \partial_e = \bigcap_{n\in\bn} A_n \).
The collection \( \{A_n \cap U\}_{n\in\bn} \) gives rise to a countable neighborhood basis of \( e \) in \( \ce(U) \) and hence \( e \) is a short end of \( U \).

We have established that \( \ce(U) \) is second countable and every end of \( U \) is short. Furthermore, notice that if $\mathcal E(M)$ is countable, then so is $\mathcal E(U)$ (as the former is a quotient of the latter, where all but countably many fibers of the quotient map are singletons, and the rest are finite sets). On the other hand, if $M$ is of type I then so is $U$ (if $\langle M_\alpha\big|\alpha<\omega_1\rangle$ is a canonical sequence for $M$, then $\langle M_\alpha\cap U\big|\alpha<\omega_1\rangle$ is a canonical sequence for $U$). Hence, either $U$ is of type I or $\mathcal E(U)$ is countable, and so we can apply Proposition~\ref{prop:sep-omega} to conclude that \( U \) is Lindel\"of.
This establishes (1).

For (2), let \( U \) be the given Lindel\"of subset so that \( M \ssm \cl_M(U) = \bigsqcup_{i \in I} T_i \), where \( T_i \) is a long trunk and \( I \) indexes the components of \( M \ssm \cl_M(U) \) (by assumption, \( I \) is countable). If \( I = \varnothing \), then \( M \) is Lindel\"of---hence of type I---and it follows from Proposition~\ref{prop:ends} that the end space of \( M \) is second countable and every end is short.  We will now assume that \( I \) is non-empty. Using Lemma~\ref{lem:connected}, for each \( i \in I \), we can find a compact bordered manifold \( K_i \) of \( T_i \) such that \( T_i\ssm K_i \) has a component \( V_i \) with \( \hat V_i \) a singleton consisting of the unique long end of \( T_i \).
By definition, all components of \( T_i \ssm K_i \) other than \( V \) must be Lindel\"of.

It now follows that \( N= M \ssm \left( \bigcup_{i \in I} \cl_M(V_i) \right) \) is a Lindel\"of set (as it is the union of \( U \) together with countably many Lindel\"of subsets). Note that \( N \) is open and captures every end of \( M \). 
In particular, by Corollary~\ref{cor:nice-subspace}, \( \ce(M) \) is homeomorphic to a quotient of \( \ce(N) \) and hence is second countable.

We will now argue that $M$ is of type I. Since each $T_i$ is a long trunk, and in particular of type I, we may take canonical sequences $\langle M_\alpha^i\big|\alpha<\omega_1\rangle$ for $T_i$. 
We then define $M_\alpha=N\cup\left(\bigcup_{i\in I}M_\alpha^i\right)$  for $\alpha<\omega_1$ to obtain the canonical sequence $\langle M_\alpha\big|\alpha<\omega_1\rangle$ witnessing that $M$ is of type I.

It is left to show that \( M \) has the EDP.
Let \( \pi_N \co \cf(M) \to \cf_M(N) \) be the continuous surjection given by Lemma~\ref{lem:relative-projection}.
Now, as every point of \( \cf(N) \) is a \( G_\delta \) point, every point of \( \cf_M(N) \) is also a \( G_\delta \) point.
The pre-image of a \( G_\delta \) subset is again a \( G_\delta \) subset; hence, if \( e \) is an end of \( M \) such that \( \{ e \} = \pi_N^{-1}[\{\pi_N(e)\}] \), then \( e \) is short.

Now, let \( e \) be an end of \( M \) such that \( \pi_N^{-1}[\{\pi_N(e)\}] \neq \{e\} \).
In this case, we must have that \( \pi_N^{-1}[\{\pi_N(e)\}] = \overline{ V}_i \) for some \( i \in I \). 
There is a unique end in \( \overline{V}_i \) and it is long; hence, \( e \) is long.
We can conclude that \( M \) has the EDP.
\end{proof}

\subsection{The Bagpipe Theorem}

We now recall Nyikos's ``Bagpipe Theorem" \cite[Theorem 5.14]{NyikosTheory}.
Before doing so, we need to introduce the notion of a bagpipe.

\begin{Def}
\label{def:pipe}
A topological space \( D \) is a \textbf{long plane} if it can be written as a union \( \bigcup_{\al < \omega_1} D_\al \) of open subspaces \( D_\al \) satisfying:

\begin{itemize}
\item  \( D_\al \) is homeomorphic to \( \br^2 \),
\item \( \cl_D( D_\al) \subset D_\be \), and
\item the boundary of \( D_\al \) in \( D_\be \) is homeomorphic to \( \bS^1 \) (the unit circle)
\end{itemize}
whenever \( \al < \be < \omega_1 \).
A space \( A \) is a \textbf{bordered long pipe} if it can be realized as a long plane with the interior of a copy of a closed disk removed, or equivalently, \( A \) is a bordered long pipe if it can be written as a union \( \bigcup_{\al<\omega_1} A_\al \) of open subspaces \( A_\al \) satisfying:
\begin{itemize}
\item  \( A_\al \) is homeomorphic to \( \bS^1 \times [0, \infty) \) (where  \( [0,\infty) \subset \br \)),
\item \( \cl_A( A_\al) \subset A_\be \), and
\item the boundary of \( A_\al \) in \( A_\be \) is homeomorphic to \( \bS^1 \)
\end{itemize}
whenever $\al<\be<\omega_1$. 
A space \( P \) is a \textbf{long pipe} if it can be realized as the manifold points of a bordered long pipe, or equivalently, if it can be realized as a long plane with a point deleted.
\end{Def}

The definition of long pipe given is equivalent to the definition given by Nyikos \cite[Definition~5.2]{NyikosTheory}.
Our bordered long pipe is what Gauld refers to as a long pipe in \cite[Definition 4.11]{GauldNonmetrisable} and hence Gauld's long pipe is not a manifold. 
Note that a long pipe is two-ended with one short and one long end. 

\begin{Ex}
Following the notation in Example~\ref{ex:long-line},  \( \bS^1 \times \mathbb L^{\geq 0} \) is a bordered long pipe, \( \bS^1 \times \mathbb L^+ \) is a long pipe, and \( \bS^1 \times \mathbb L \) is neither.
Additionally, note that \( \mathbb L \times \mathbb L \) is a long plane while \( \mathbb L^+\times \mathbb L^+ \) and \( \mathbb L \times \mathbb L^+ \) are not. 
In fact, the long planes are exactly the \( \omega \)-bounded simply-connected 2-manifolds (see \cite[Lemma 6.1]{NyikosTheory}), or, equivalently, using Theorem~\ref{lem:finite-long}, the long planes are exactly the simply-connected 2-manifolds whose unique end is long.
Unfortunately, these simple examples hide the complexity of the situation:  there are exactly \( 2^{\aleph_1} \) many pairwise non-homeomorphic long planes (see Gauld \cite[Theorem 4.19]{GauldNonmetrisable}).
\end{Ex}

Following Nyikos's terminology, we will use the term \emph{surface} to refer to a 2-dimensional bordered manifold.
It therefore makes sense to discuss compact subsurfaces of a surface. 
It is worth noting that every boundary component of a compact surface is homeomorphic to the circle; it is this rigidity that allows for the strengthening of the Bagpipe Lemma to the Bagpipe Theorem given below in dimension two.

\begin{Thm}[``The Bagpipe Theorem"{\cite[Theorem 5.14]{NyikosTheory}}]
Every \( \omega \)-bounded 2-manifold \( M \) has a compact subsurface \( K \) such that \( M \ssm K \) is the union of finitely many disjoint long pipes.
\end{Thm}

The title of the theorem stems from the following definition and corollary:

\begin{figure}[t]
\includegraphics[scale=0.9]{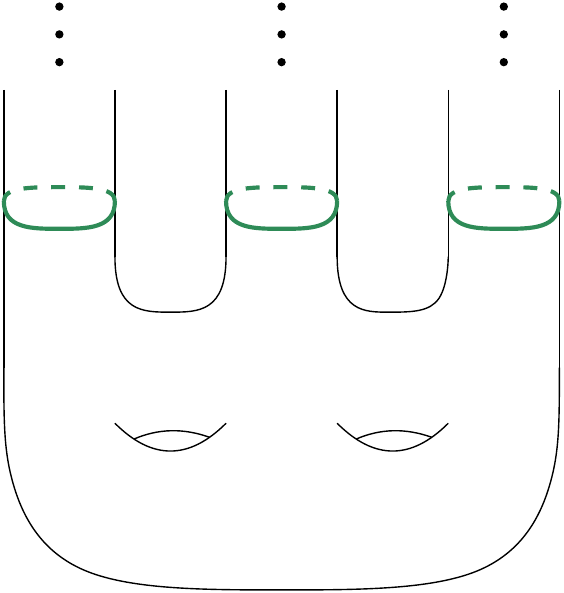}
\caption{A 3-ended genus-2 bagpipe.  Each of the simple closed curves shown bounds a long pipe.  The compact surface co-bounded by the simple closed curves drawn is the \emph{bag}.}
\label{fig:bagpipe}
\end{figure}

\begin{Def}
A \textbf{bagpipe} is the connected sum of a 2-sphere and finitely many tori, projective planes, and long planes (equivalently, a bagpipe is the connected sum of a compact 2-manifold with finitely many long planes).  See Figure~\ref{fig:bagpipe}.
\end{Def}

We note that the reason we introduce long planes (and Gauld uses an alternate definition of long pipe) is to be able to make the above definition of a bagpipe.  

\begin{Cor}
A 2-manifold is \( \omega \)-bounded if and only if it is a bagpipe.
\end{Cor}

We broaden the definition of a bagpipe by allowing the ``bag" to be second countable, rather than compact.
First, let a \emph{bordered bagpipe} be the connected sum of a compact surface with finitely many long planes\footnote{As is the case with bordered manifolds, it is possible for a bordered bagpipe to have empty boundary.}.

\begin{Def}
A \textbf{general bagpipe} is a space that can be written as a union \( \bigcup_{n \in \bn} B_n \) of closed subsets \( B_n \) such that \( B_n \) is a bordered bagpipe and \( B_n\subseteq\int(B_{n+1}) \) for all \( n \in \bn \). Equivalently, a general bagpipe is a 2-manifold constructed as the connected sum of a connected open subset of the 2-sphere and countably many tori, projective planes, and long planes. See Figure~\ref{fig:general_bagpipe}.
\end{Def}

The goal of this section is to prove:

\begin{Thm}[``The General Bagpipe Theorem"]
\label{thm:big-bag}
Let \( M \) be a  2-manifold.
\begin{enumerate}
\item
If \( M \) satisfies the EDP and either (i) \( M \) is of type I and has second countable end space or (ii) \( M \) has countable end space, then there exists an open  Lindel\"of subset \( U \) of \( M \) such that \( \cl_M(U) \) is a surface with compact boundary components and \( M \ssm U \) is the disjoint union of countably many bordered long pipes.
\item
If there exists an open  Lindel\"of subset \( U \) of \( M \) such that \( M \ssm \cl_M(U) \) is the disjoint union of countably many long pipes, then \( M \) satisfies the EDP, has second countable end space, and is of type I.
\end{enumerate} 
\end{Thm}

\begin{figure}[t]
\includegraphics[scale=0.9]{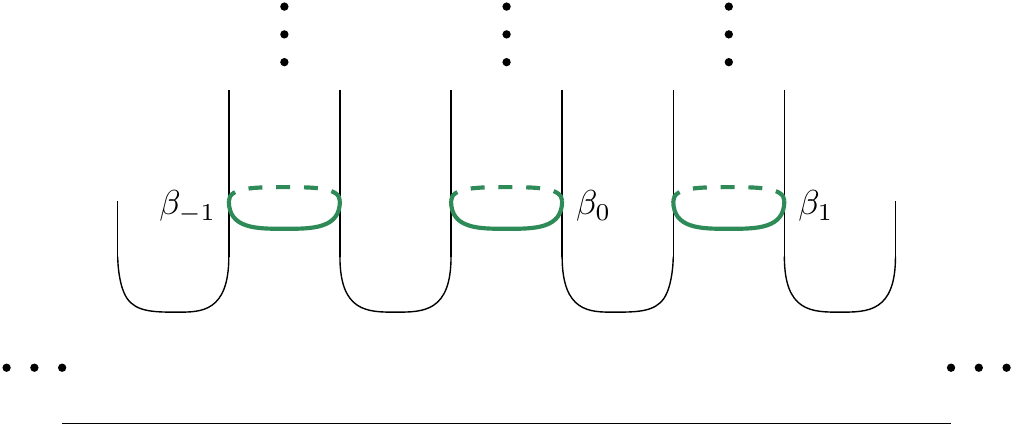}
\caption{A general bagpipe with two short ends---the left and right horizontal endpoints---and a countably infinite set of long ends, which belong to long pipes bounded by the \( \beta_n \).
Observe that the subsurface co-bounded by the \( \beta_n \) is second countable. 
}
\label{fig:general_bagpipe}
\end{figure}

\begin{proof}
First note that (2) follows immediately from Theorem~\ref{thm:gen-bag-lemma} as long pipes are long trunks.
For (1), we begin by invoking Theorem~\ref{thm:gen-bag-lemma}, which yields a Lindel\"of subsurface \( U' \) of \( M \) such that \( \cl(U') \) is a bordered manifold, \( \partial(U') \) is bi-collared, and \( M \ssm \cl_M(U') = \bigsqcup_{i\in I} T_i \), where \(I\) is countable and \( T_i \) is a long trunk.
Moreover, note that \( \cl_M(T_i) \) is a surface for each \( i \in I \). Use Lemma~\ref{lem:connected} to shrink $T_i$ so that \( \cl_M(T_i) \) is a bordered manifold with bi-collared boundary in \( M \)  such that $\hat{T_i}$ only contains the unique end of $T_i$. In this case \( \cl_M(T_i) \) is an \( \omega \)-bounded surface (by Theorem~\ref{lem:finite-long}) and hence, by \cite[Corollary 5.16]{NyikosTheory} (a corollary of the Bagpipe Theorem), there exists a compact subsurface \( K_i \) of \( T_i \) such that \( \cl_M(T_i \ssm K_i) \) is a bordered long pipe.

Define \( U \) to be the interior of \( U' \cup \left( \bigcup_{i \in I} K_i \right) \).
As \( U \) is the interior of a Lindel\"of subset of a manifold, it itself is Lindel\"of.
By construction, the complement of \( \cl_M(U) \) is a disjoint union of countably many long pipes.
\end{proof}

\begin{Cor}
Let \( M \) be a 2-manifold of type I.
The manifold \( M \) satisfies the EDP and has second countable end space if and only if it is a general bagpipe.
\end{Cor}

\begin{proof}
Let us first assume that \( M \) is a general bagpipe.
Write \( M = \bigcup_{n\in\bn} B_n \), where \( B_n \) is a bordered bagpipe, \( B_n \) is a closed subset of \( M \), and \( B_n \subset \int(B_{n+1}) \). 
The inclusion \( \iota_n \co B_n \hookrightarrow M \) is proper and induces an embedding \( \hat \iota_n \co \ce(B_n) \to \ce(M) \); let \( E_n \) denote the image of \( \hat \iota_n \).
By Theorem~\ref{lem:finite-long}, \( B_n \) has finitely many ends, all of which are long; hence, \( E_n \) is finite and consists of long ends. 
Let \( E_\infty \) denote the complement of \( \bigcup_{n\in\bn} E_n \) in \( \ce(M) \). 
By construction, if \( E_\infty \neq \varnothing \), then it is clear that every end of \( E_\infty \) is short and so \( \ce(M) \) has the EDP.
In addition, the collection \( \{\{e\} \mid e\in E_n\text{ for some }n\in\bn\} \) together with the sets of the form \( \hat U_n \), where \( U_n \) is a component of \( M \ssm B_n \), form a countable basis for the topology of \( \ce(M) \); hence, \( \ce(M) \) is second countable. 

For the other direction, assume \( M \) satisfies the EDP and has second countable end space.
By definition, a compact 2-manifold is a bagpipe, so let us assume that \( M \) is non-compact. 
If \( M \) only has long ends, then, by Theorem~\ref{lem:finite-long}, \( M \) is \( \omega \)-bounded and hence a bagpipe. 
So, let us assume that \( M \) has at least one short end. 
Then, by the General Bagpipe Theorem, there is an open subset \( U \) of \( M \) such that \( \Sigma = \cl_M(U) \) is a second-countable surface with compact boundary components and \( M \ssm U \) is the disjoint union of countably many bordered long pipes.

As we have seen before, as every long end is isolated (Lemma~\ref{lem:isolated}) and \( \ce(M) \) is second countable, there are at most countably many long ends; let \( \{\ell_i\}_{i\in I} \) be an enumeration.
By construction, for each \( i \in I \), there is a unique component of \( M \ssm \Sigma \), call it \( L_i \), such that \( \ell_i \in \hat L_i \) (note that \( L_i \) is a long pipe).
We also let \( \be_i = \partial (L_i) \).  

Now, as \( \Sigma \) is a non-compact second-countable surface there exists a collection of compact subsurfaces \( \{K_n\}_{n\in\bn} \) such that \( \Sigma = \bigcup_{n\in\bn} K_n \) and \( K_n \subset \int(K_{n+1}) \). 
In addition, for each \( n \in \bn \), if \( I \) is finite, we require \( K_n \) to contain \( \beta_i \) for all \( i \in I \); otherwise, we identify \( I \) with \( \bn \) and require \( K_n \) to contain \( \be_i \) for all \( i \leq n \). 
In either case, let \( B_n \) be obtained by taking the union of \( K_n \) with every \( L_i \) satisfying \( \be_i \subset K_n \). 
We then have---as desired---that \( B_n \) is closed subset of \( M \), \( B_n \subset \int(B_{n+1}) \), \( B_n \) is a bordered bagpipe, and \( M = \bigcup_{n\in\bn} B_n \); hence, \( M \) is a general bagpipe.
\end{proof}

\appendix

\section{A type I manifold with all short ends and \\ whose end space is not second countable}
\label{sectionexamples}

In this appendix, we build a type I manifold, all of whose ends are short and whose end space fails to be second countable. Consequently, the manifold fails to be a general bagpipe. This shows that second countability of the end space is a necessary assumption in Theorem~\ref{thm:gen-bag-lemma} and Theorem~\ref{thm:big-bag}. This manifold is built using (set-theoretic) trees, so we first proceed to remind the reader of some standard facts about trees.

\begin{Def}
A {\bf tree} is a partially ordered set $(T,\leq)$ such that, for every $t\in T$, the set $\{s\in T\big|s\leq t\}$ is (with the inherited order) well-ordered.
\end{Def}

All of the trees used here will be {\it connected} (i.e., every tree will be assumed to have a minimum element, called its {\it root}). 
An element \( t \in T \) is said to have {\em height} \( \al \), written \( \h(t) = \alpha \), if \( \{s \in T\big| s<t\} \) has order-type \( \al \).
Given an ordinal $\alpha$, the $\alpha$-th level of the tree is the set
\begin{equation*}
T_\alpha=\{t\in T\big|\h(t) = \al\},
\end{equation*}
and we will use notations such as $T_{\leq\alpha}$ or $T_{<\alpha}$ with the obvious meanings. The {\em height} of the tree $T$, denoted $\h(T)$, is the least $\alpha$ such that $T_\alpha=\varnothing$. An {\em $\omega_1$-tree} is a tree of height $\omega_1$ that has countable levels. 
A {\em chain} of the tree $T$ is a subset $C\subseteq T$ any two of whose elements are $\leq$-comparable, and an {\em antichain} of $T$ is a subset $A\subseteq T$, any two of whose elements are $\leq$-incomparable. A maximal chain is called a {\em branch}. Also, elements of $T$ are often called {\em nodes}.

We will moreover assume that every tree under consideration is {\it well-pruned}, that is, for every node $t\in T$ and every $\alpha$ such that $\h(t)<\alpha<\h(T)$, there exists a node $s\in T_\alpha$ with $t\leq s$. We will also assume every tree to be {\it Hausdorff}, i.e., for every limit ordinal $\alpha$ and every chain $C$ such that $\sup\{\h(t)\big|t\in C\}=\alpha$, there exists at most one $t\in T_\alpha$ whose set of predecessors contains $C$. Finally, all of our trees will satisfy that, whenever a node $t$ of the tree has at least one successor, then it has at least two immediate successors (that is, there are at least two distinct $s_1,s_2\geq t$ such that $\h(s_1)=\h(s_2)=\h(t)+1$).

The following definition specifies the kind of tree that we will need for our construction.

\begin{Def}
An {\bf Aronszajn tree} is an $\omega_1$-tree without uncountable branches.
\end{Def}

Aronszajn trees are counterexamples to the natural analog of K\"onig's lemma at cardinality $\aleph_1$ (recall that K\"onig's lemma states that every tree of height $\omega$ with finite levels must have an infinite branch). The following is a quite classical result in set theory.

\begin{Thm}\label{existsaronszajn}
There exists an Aronszajn tree.
\end{Thm}

\begin{proof}
See~\cite[Theorem II.5.9, p. 70]{KunenSetTheory} or~\cite[Theorem III.1.1, p. 111]{DevlinConstructibility}.
\end{proof}

We point out that, in fact, there exists a binary Aronszajn tree (i.e., an Aronszajn tree each of whose nodes has exactly two immediate successors); such a tree can be obtained by taking an arbitrary Aronszajn tree and recursively removing nodes at each level to ensure that every remaining node has exactly two immediate successors (the resulting tree will still be Aronszajn).

\begin{Def}\label{treetopology}
Given a tree $T$, we will define a topology on $T$ as follows. Let $S(T)$ be the collection of nodes whose height is a successor ordinal. We stipulate that the collection
\begin{equation*}
\{t^\uparrow \mid t\in S(T)\} \cup \{ T \ssm t^{\uparrow} \mid t\in S(T)\}
\end{equation*} 
forms a subbasis---of clopen sets---for the topology, where $t^{\uparrow}=\{x\in T\mid t\leq x\}$ is the ``upward cone'' with base $t$.
\end{Def}

Note that, in Definition~\ref{treetopology}, every node whose height is a successor ordinal will be isolated. On the other hand, nodes of limit height are always accumulation points of their set of predecessors.

In order to state the next theorem, we introduce some terminology.

\begin{Def}\label{treeclosure}
Let $T$ be a tree.
\begin{enumerate}
\item A chain $C\subseteq T$ will be called {\bf full} if it is closed downwards. That is, if $t\in C$ and $s\leq t$ then $s\in C$.
\item A chain $C\subseteq T$ will be called {\bf closed} if it is a closed set in the tree topology. Equivalently, whenever $\alpha$ is a limit ordinal and $t\in T_\alpha$ is such that, for unboundedly many $\xi<\alpha$, the predecessor $t_\xi$ of $t$ at level $\xi$ belongs to $C$, then $t\in C$.
\item We define a new tree, denoted $\overline{T}$, whose nodes are all closed and full chains of $T$, ordered by inclusion. That is, 
\begin{equation*}
\overline{T}=\{C\subseteq T\big|C\text{is a closed full chain}\},
\end{equation*}
with ordering given by $C\leq C'$ if and only if $C\subseteq C'$.
\end{enumerate}
\end{Def}

Note that, according to Definition~\ref{treeclosure}, we can always embed the tree $T$ into the tree $\overline{T}$ by mapping each $t\in T$ to the closed and full chain  $C_t = \{s\in T\big|s\leq t\}$. Thus, $\overline{T}$ contains an isomorphic copy of $T$; in fact, $\overline{T}$ can be thought of as the tree that results from taking $T$ and, for every branch $C$ in $T$ without a maximal element, adding a node on top of $C$.
Going forward, we abuse notation and identify \( T \) with its isomorphic copy in \( \overline T \).

\begin{Thm}\label{treeconstruction}
Let $T$ be a binary $\omega_1$-tree. Then there exists a  2-manifold of type I  whose space of ends is homeomorphic to the topological space arising from $\overline{T}$. Furthermore, all ends corresponding to nodes in $\overline{T}_{<\omega_1}$ are short.
\end{Thm}

\begin{proof}
The construction employs a number of ``building blocks''. Each of these building blocks is constructed as follows:
let \( P \) be a compact pair of pants, that is, a compact, orientable surface of genus 0 with three boundary components (e.g., the 2-sphere with three pairwise-disjoint open disks removed).
Let \( B \) be obtained by removing a point from each boundary component of \( P \).
Note that \( B \) is a non-compact surface with three boundary components each of which is homeomorphic to \( \br \). 
Also note that \( B \) has three ends.

We now simultaneously define, by recursion on $\alpha<\omega_1$, (1)  orientable Lindel\"of bordered surfaces, denoted $S_\alpha$, along with bijections $\varphi_\alpha$ between the nodes of the $(\alpha+1)$-st level $T_{\alpha+1}$ of the tree $T$ and the boundary components of $S_\alpha$, and (2) functions \( f_\alpha \co \overline {T}_{\leq(\alpha+1)} \to \ce(S_\alpha) \). 
We will require that each boundary component of $S_\alpha$ be homeomorphic to \( \br \).

To begin, let \( S_0 = B \ssm \partial_1 \), where \( \partial_1 \) is a boundary component of \( B \).  
We let $\varphi_0$ be any bijection between the two nodes that belong to the level $T_1$ and the two boundary components of $S_0$ (intuitively speaking, the boundary component that we deleted corresponds to the root of the tree, and $\varphi_{-1}$ should be the bijection mapping this ``hole'' to the root of $T$). 
Now there is a unique end of \( S_0 \) that can only be reached by manifold points of \( S_0 \), call it \( e_0 \) (and note that $e_0$ is short), and define \( f_0(T_0) = e_0 \).  
Next, let \( b \) be a boundary component of \( S_0 \), then there is a unique end of \( S_0 \), call it \( e_b \), that is in the closure of \( b \) in \( \cf(S_0) \) (and note that this end $e_b$ must be short); define \( f_0(\varphi_0^{-1}(b)) = e_b \). 

Now assume that $S_\alpha$, $\varphi_\alpha$, and \( f_\alpha \) have been constructed. To define $S_{\alpha+1}$, we first pick, for each node $t\in T_{\alpha+1}$ of the tree $T$ at level $\alpha+1$, a copy $B_t$ of the basic building block $B$, in such a way that these copies are all pairwise disjoint. We now proceed to glue $S_\alpha$ with all of the $B_t$ by identifying one of the boundary components of $B_t$ with the boundary component $\varphi_\alpha(t)$ of $S_\alpha$ via orientation-reversing homeomorphisms; the result of this gluing process will be our $S_{\alpha+1}$. Next, we build the mapping $\varphi_{\alpha+1}$ by mapping each of the two immediate successors $t_1,t_2$ of the node $t$ to the two remaining boundary components of $B_t$. 
Then, as before, given a boundary component \( b \) of \( S_{\al+1} \), there is a unique end \( e_b \) of \( S_{\al+1} \) in the closure of \( b \) in \( \cf(S_{\al+1}) \) (and, once again, notice that the end $e_b$ is short): define \( f_{\al+1}(\varphi_{\alpha+1}^{-1}(b)) = e_b \). 
Finally, as \( S_\al \) is a closed subset of \( S_{\al+1} \), the inclusion map \( S_\al \hookrightarrow S_{\al+1} \) is proper and hence induces a continuous map \( \iota_\al \co \ce(S_\al) \to \ce(S_{\al+1}) \); moreover, it is not hard to see that \( \iota_\al \) is an embedding.
To finish, we require \( f_{\al+1}\circ \iota_\al = \iota_\al \circ f_\al (t) \). 

\begin{figure}[t]
\includegraphics{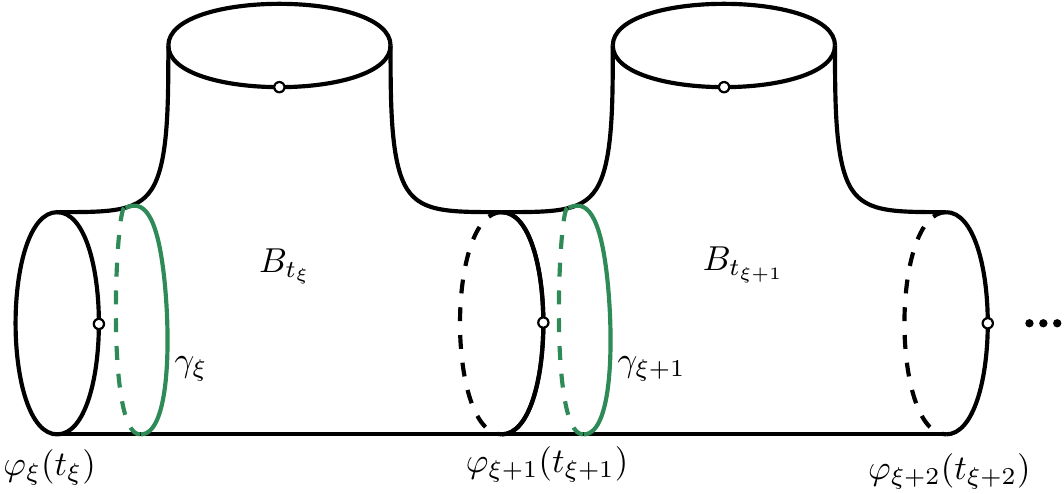}
\caption{Two consecutive copies of \( B \) glued at stage \( \xi \) and \( \xi~+~1 \), with the curves \( \gamma_\xi \) and \( \gamma_{\xi+1} \) as used in the construction of the homeomorphism \( \Phi \co \overline{T}_{\leq \al} \to \ce(X) \).} 
\label{fig:tree}
\end{figure}

It remains to describe the construction of $S_\alpha$, $\varphi_\alpha$, and \( f_\al \) when $\alpha$ is a limit ordinal and all of the $S_\xi$, $\varphi_\xi$, and \( f_\xi \), for $\xi<\alpha$, have already been constructed. 
First let $X=\bigcup_{\xi<\alpha}S_\xi$ and equip \( X \) with the direct limit topology.
We claim that there is a homeomorphism $\Phi:\overline{T}_{\leq\alpha}\rightarrow\mathcal E(X)$. 
Given our setup, we can define \( \displaystyle{\varinjlim_{\xi < \al} \ce(S_\xi)} \) and see that it embeds in \( \ce(X) \), let \( \nu \) denote this embedding.
For \( t \in \overline T_\xi \) with \( \xi < \alpha \), we define \( \Phi(t) = \nu \circ f_\xi(t) \). 
We now need to define \( \Phi \) for \( t \in \overline T_\al \). 
Fix \( t \in \overline T_\al \) and, for each $\xi<\alpha$, let $t_\xi$ be the predecessor of $t$ at level $\xi$. Then  \( B_{t_\xi}  \subset S_{\xi} \) contains three boundary components; one of them corresponds to \( \varphi_{\xi}(t_\xi) \) and another corresponds to \( \varphi_{\xi+1}(t_{\xi+1}) \). Let \( \gamma_{t_\xi} \) be a simple closed curve contained in $B_{t_\xi}$ separating the boundary component of \( B_{t_\xi} \) corresponding to $\varphi_\xi(t_\xi)$ from the other two boundary components of \( B_{t_\xi} \), and set \( V_{t_{\xi+1}} \) to be the component of \( X \ssm \gamma_{t_\xi} \) containing \( \varphi_{\xi+1}(t_{\xi+1}) \) (see Figure~\ref{fig:tree}). 
Let \( \mathcal V = \{ V_{t_{\xi+1}} | \xi<\alpha \} \) and observe that if \( V \in \mathcal V \) then V is unbounded and \( \partial(V) \) is compact. 
Moreover, \( \bigcap_{V\in \mathcal V} V = \varnothing \); hence, \( \hat{\mathcal V} = \{ \hat V | V \in \mathcal V\} \) is a neighborhood basis for a unique end of \( X \), call it \( e_t \) (notice, in particular, that the fact that $\bigcap_{V\in\mathcal V}\hat{V}=\{e_t\}$ implies that $e_t$ is a short end). We define \( \Phi(t) = e_t \). 
Now observe that if \( t \in T_{\xi+1} \) for \( \xi < \al \), then \( \Phi[t^\uparrow] =\hat{V_t} \).
Note that both \( t^\uparrow \) and \( \hat{V_t} \) are clopen, so we may conclude that \( \Phi \co \overline T_{\leq\al} \to \ce(X) \) is an open bijection; in particular, \( \Phi^{-1} \) is a continuous bijection between compact Hausdorff spaces and hence a homeomorphism. 

Now observe that \( X \) is an orientable second-countable 2-manifold of genus 0; in particular, by the classification of second-countable surfaces, \( X \) is homeomorphic to an open subset of the 2-sphere.
We will now give an explicit construction of such an open subset: for the simplicity of using coordinates, let us identify the 2-sphere with the Riemann sphere \( \hat{\mathbb{C}} \) and let \( E \subset \br \) be a second-countable Stone space homeomorphic to \( \ce(X) \).
Let \( \phi \co \ce(X) \to E \) be a homeomorphism and  let $N=(\phi\circ\Phi)[T_\alpha]\subseteq E$. In other words, $N$ is the set of ends of $X$ that correspond to nodes in $T_\alpha$. 

By assumption, \( N \) is countably infinite: 
choose an enumeration \( N = \{ x_n\}_{n\in\bn} \). 
Inductively, choose a collection of Euclidean circles \( \{C_n\}_{n\in\bn} \) such that, for each \( n \in \bn \), \( C_n \cap \br = \{ x_n \} \) and such that the radius \( r_n \) of \( C_n \) satisfies \( r_n < \frac1n \) and such that \( r_n+r_m < |x_n - x_m| \) for \( 1 \leq m < n \).
Define \( F \subset \hat{ \mathbb C} \) to be the complement of 
\[
E \cup \bigcup_{n\in\bn} D_n,
\]
where \( D_n \) is the open disk in \( \mathbb C \) bounded by \( C_n \). 
Note that \( F \) is a bordered manifold in which every boundary component is homeomorphic to \( \br \). 

By the classification of surfaces, the set of manifold points of \( F \) is a manifold homeomorphic to \( X \). By construction, each node $t\in T_\alpha$ is naturally associated to a boundary component $b_t$ of \( F \). Pick pairwise-disjoint copies $B_t$ of the basic building block $B$ (as $t$ varies over $T_\alpha$), and proceed to glue the $B_t$ to $F$ by identifying a boundary component of $B_t$ with the boundary component $b_t$ of $F$ via an orientation-reversing homeomorphism. The resulting manifold will be $S_\alpha$. For each $t\in T_\alpha$, there are two boundary components $b_1,b_2$ of $S_\alpha$ that naturally correspond to $t$, namely the two remaining boundary components of the basic building block $B_t$. If $t_1,t_2$ are the two immediate successors of $t$, we define $\varphi_\alpha(t_1)=b_1$ and $\varphi_\alpha(t_2)=b_2$. Doing this for all $t\in T_\alpha$ finishes the definition of $\varphi_\alpha$. Finally, we proceed to build $f_\alpha$. First notice that $\overline{T}_{\leq(\alpha+1)}=\overline{T}_{\leq\alpha}\cup T_{\alpha+1}$; notice also that for each boundary component $b$ of $S_\al$ there is a unique end $e_b$ that lies in the closure of $b$ in $\mathcal F(S_\al)$ (furthermore, notice once again that $e_b$ is a short end). Now, we have a natural embedding \( \iota:\ce(X) \to \ce(S_\al) \) (induced by the homeomorphism between $X$ and $F$), and every element of $\ce(S_\al)$ that is not in the image of $\iota$ must be one of the $e_b$. We have already constructed a homeomorphism \( \Phi \co \overline{T}_{\leq \al} \to \ce(X) \), so for $t\in\overline{T}_{\leq\alpha}$ we define $f_\alpha(t)=\iota\circ\Phi(t)$, and for $t\in T_{\alpha+1}$ we let $f_\al(\varphi_\alpha^{-1}(b)) = e_b$. This finishes the definition of $f_\alpha$, and hence of the inductive step at the limit ordinal $\alpha$.

We let $M=\bigcup_{\alpha<\omega_1}S_\alpha$, equipped with the direct limit topology. This 2-manifold is of type I, as witnessed by the sequence $\langle U_\alpha\big|\alpha<\omega_1\rangle$, where $U_\alpha=S_\alpha\setminus\partial(S_\alpha)$. 
The construction of the homeomorphism \( \overline T\rightarrow \ce(M) \) is identical to the previous construction of \( \Phi \) above for a limit ordinal. The only difference with the case of a limit $\alpha<\omega_1$ is that, for any node $t\in\overline{T}_{\omega_1}$, the collection $\mathcal V=\{V_{t_{\xi+1}}\big|\xi<\omega_1\}$ is no longer countable and hence we can no longer ensure that $\Phi(t)$ is a short end.
\end{proof}

\begin{Thm}
There exists a 2-manifold of type I, with all ends short, that is not a general bagpipe.
\end{Thm}

\begin{proof}
By Theorem~\ref{existsaronszajn}, let $T$ be an Aronszajn tree, and let $M$ be the type I surface built from $T$ as in Theorem~\ref{treeconstruction}. Then there is a homeomorphism $\Phi:\overline{T}\rightarrow\mathcal E(M)$, such that, whenever $t\in\overline{T}_{<\omega_1}$, the end $\Phi(t)$ is short.

Since $T$ is Aronszajn, by definition there are no uncountable branches of $T$. That is, if $C\subseteq T$ is a branch, then $C$ is countable, and consequently $\sup\{\h(t)\big|t\in C\}$ is a countable ordinal. This implies that every node in $\overline{T}$ has a countable height, in other words, $\overline{T}_{\omega_1}=\varnothing$ and so $\overline{T}=\overline{T}_{<\omega_1}$. In particular, every end of $M$ (being of the form $\Phi(t)$ for $t\in\overline{T}_{<\omega_1}$) is short.

Now suppose that $X\subseteq\overline{T}$ is a countable set. Since every node of $\overline{T}$ has a countable height, it follows that $\sup\{\h(t)\big|t\in X\}$ must be a countable ordinal; in other words, there is an $\alpha<\omega_1$ such that $X\subseteq\overline{T}_{\leq\alpha}$. Then every node of height $\alpha+1$ fails to be in the closure of $X$ and so $X$ cannot be dense in $\overline{T}$. Hence, $\overline{T}$---and therefore also $\mathcal E(M)$---is not separable, and so it is not second countable either. Then, by Theorem~\ref{thm:big-bag}, $M$ cannot possibly be a general bagpipe as $\mathcal E(M)$ fails to be second countable.
\end{proof}

\section{Pr\"uferization, Mooreization, and spaces of ends}
\label{appendix-prufer}

\begin{center}
Mathieu Baillif, David Fern\'andez-Bret\'on, and Nicholas G.~Vlamis
\end{center}

In this appendix we explore the Pr\"uferization process, and some of the spaces of ends that one can get by Pr\"uferizing distinct sets of points, as well as by doubling or otherwise manipulating the resulting manifolds. 
For this section, we provide sketches of all relevant arguments and leave the details to the reader\footnote{Specifically regarding the details of Pr\"uferizations and Mooreizations, the reader can consult~\cite[Section 1.3]{GauldNonmetrisable}.}. Along the way, we establish the existence of a non-metrizable manifold, all of whose ends are short, with a second-countable end space (Theorem~\ref{thm:counterexample} below); this shows that the type I hypothesis is necessary in Theorems~\ref{mainthm1},~\ref{mainthm2}, and~\ref{mainthm3}, as well as in our characterization of metrizability, Theorem~\ref{thm:characterize} (ii) and Corollary~\ref{cor:characterize-typeI}.

We first describe what it means to Pr\"uferize one point. Consider the open half-plane $H = \{ (x,y)\in \br^2 \big| x>0 \}$ and let $p = (0,y_p)$ for some \( y_p \in \br \).
The underlying set of the Pr\"uferization $P_p(H)$ of $H$ at $p$ is $H\cup\mathbb R_p$, where $\mathbb R_p$ is a copy of the real line. Now, 
given \( a,b \in \br_p \) such that \( a< b \), and given $c\in\mathbb R^+$, we define \( U_{a,b,c} \subset P_p(H) \)  by
\[ U_{a,b,c} = \{ z \in \br_p \big| a < z < b\}  \cup \{ (x,y) \in H \big| y < bx+ y_p \text{, } y >ax+y_p\text{, and }x<c \}. \] 
The topology of $P_p(H)$ is generated by basic open sets of two kinds:
\begin{enumerate}
\item the standard (Euclidean) open sets in $H$, and
\item the sets \( U_{a,b,c} \) for every $a<b$ in $\mathbb R_p$ and $c\in\mathbb R^+$.
\end{enumerate}
This has the effect that every element of $\mathbb R_p$ is ``close'' to the point where $p=(0,y_p)$ ``would be located''. A neighborhood basis of the point $z\in\mathbb R_p$ is given by the sets $U_{a,b,c}$ with $a<p<b$
, which intersect $H$ in a small triangle enclosing a line segment of slope $z$ that passes through $p$ (e.g. if $p$ is the origin, then the sequence of points $\langle\left(\frac{1}{n},\frac{z}{n}\right)\big|n\in\mathbb N\rangle \subset H$ converges to $z$ in \( P_p(H) \)).

In order to determine what $\mathcal E(P_p(H))$ is, we need to investigate what kinds of compact subsets the space $P_p(H)$ can have. Note that the sets \( U_{a,b,c} \) are bounded and, furthermore, every compact subset of $P_p(H)$ is either contained in the half-plane (and thus it is an Euclidean compact set) or it is contained in an open set of the form $U_{a,b,c}$. The closure of $U_{a,b,c}$ is equal to 
\[ \cl_{P_p(H)}(U_{a,b,c}) = \{ z \in \br_p \big| a \leq z \leq b\}  \cup \{ (x,y) \in H \big| y \leq bx+ y_p, y \geq ax+y_p, x \leq c \}. \] 
In particular, \( \partial(U_{a,b,c}) \) is a triangle in \( H \) with a missing vertex---the vertex would be at the point \( p=(0,y_p) \), should this point belong to the half-plane---together with the points \( a, b \in \br_p \). 
Note that the complement \( V \) of \( \cl_{P_p(H)}(U_{a,b,c}) \) intersects \( \br_p \) in two disjoint half rays, namely, \( V \cap \br_p = (-\infty, a) \cup (b, \infty) \).

Before we continue to Pr\"uferize more than a single point, it is instructive to note that \( P_p(H) \) is a bordered 2-manifold---homeomorphic to \( [0,1)\times (0,1) \)---and if we double \( P_p(H) \) along its boundary, we obtain a 2-manifold homeomorphic to \( \br^2 \).

Now suppose that we let $Y$ be an arbitrary subset of the $y$-axis and that we Pr\"uferize every single point in $Y$ to obtain the bordered manifold $P_Y(H)$---the construction in the previous paragraph describes the particular case where $Y$ is a singleton. Suppose that $Y$ has at least two points $z,w$ (with $z<w$). Given an $n\in\mathbb N$, consider the set
\begin{equation*}
T_{z,w,n}=\{n_{\mathbb R_z},(-n)_{\mathbb R_w}\}\cup\left\{(x,y)\in H\big|y=nx+z\text{ or }y=-nx+w\text{ and }x\leq\frac{w-z}{2n}\right\},
\end{equation*}
where the number $n_{\mathbb R_z}$ denotes the copy of $n$ within $\mathbb R_z$, and analogously for $(-n)_{\mathbb R_w}$. In other words, $T_{z,w,n}$ consists of the portions of the line through $z$ with slope $n$ and the line through $w$ with slope $-n$ up to their intersection, together with the points $n\in\mathbb R_z$ and $-n\in\mathbb R_w$. 
Given an unbounded open set with compact boundary, it is possible to pass to a subset $U$ whose (topological) boundary is a set of the form $T_{z,w,n}$. For such a set $U$, we have the following two possibilities:
\begin{enumerate}
\item Either $U$ is the unbounded region of the half-plane delimited by its intersection with $T_{z,w,n}$, along with the open half rays $(-\infty,n)\subseteq\mathbb R_z$ and $(-n,\infty)\subseteq\mathbb R_w$, together with the whole line $\mathbb R_x$ whenever $x\in Y$ and $x<z$ or $w<x$---we will denote this open set by  $U_{z,w,n}^\infty$, or
\item $U$ is the bounded region of the half-plane delimited by its intersection with $T_{z,w,n}$, along with the open half rays $(n,\infty)\subseteq\mathbb R_z$ and $(-\infty,-n)\subseteq\mathbb R_w$, together with the whole line $\mathbb R_x$ whenever $x\in Y$ and $z<x<w$---we will denote this open set by  $U_{z,w,n}$.
\end{enumerate}

Working now in $\mathcal E(P_Y(H))$, the set  $\hat{U}_{z,w,n}^\infty$ is a neighborhood of at least one end. In fact, the $\hat{U}_{z,w,n}^\infty$, as we let $z,w,n$ vary, form a neighborhood basis of a unique end, that we will denote by $e_\infty$. 
The set  $\hat{U}_{z,w,n}$ is also a neighborhood of at least one end; the sets $\hat{U}_{z,w,n}$ as we let $w,n$ vary while leaving $z$ fixed\footnote{This includes e.g. the case where $Y$ does not contain points between $z$ and $w$, in which case $n$ is the only variable that actually varies.}, form a neighborhood basis of a unique end, that we will denote with $e_z^+$. Informally speaking, seeing ends as ``ways to escape to infinity'', $e_z^+$ corresponds to going to $+\infty$ in $\mathbb R_z$ (seen as a subset of $P_Y(H)$). Meanwhile, the sets $\hat{U}_{z,w,n}$ as we let $z,n$ vary while leaving $w$ fixed, form a neighborhood basis of an end that we will denote with $e_w^-$ (corresponding informally to going to $-\infty$ in $\mathbb R_z$, seen as a subset of $P_Y(H)$). Furthermore, if $x\in\cl(Y)\setminus Y$, then the sets $\hat{U}_{z,w,n}$ as $z,w,n$ vary with $z<x<w$ (and $z,w\in Y$), also form a neighborhood basis of another end that we will call $e_x$. It is instructive to notice that the ends just described need not be pairwise distinct: for example, if $z,w\in Y$ and there are no $x\in Y$ with $z<x<w$, then $e_z^+=e_w^-$; similar remarks apply to combinations of these and the ends $e_x$ with $x\in\cl(Y)\setminus Y$. The sets $\hat{U}_{z,w,n}^\infty$ contain, in addition to $e_\infty$, all $e_x^+$, $e_x^-$ (for $x\in Y)$, and $e_x$ (for $x\in\cl(Y)\setminus Y$) such that $w<x$ or $x<z$; as well as $e_w^+$ and $e_z^-$. Similarly, $\hat{U}_{z,w,n}$ contains, in addition to $e_z^+$ and $e_w^-$, all $e_x^+$, $e_x^-$ (for $x\in Y$), and $e_x$ (for $x\in\cl(Y)\setminus Y$) such that $z<x<w$.

After the previous analysis, we are able to describe the space of ends $\mathcal E(P_Y(H))$ of the Pr\"uferization of the half-plane $H$ at every point in $Y$. We start by considering an end $e_\infty$ and, for each $z\in\cl(Y)$, two ends $e_z^+$ and $e_z^-$, equipped with the following topology and subject to the following identifications:
\begin{itemize}
\item If $z\in\cl(Y)\setminus Y$ then $e_z^+=e_z^-$ (and hence, we will drop the superindices in this case and simply write $e_z$).
\item The sets $\hat{U}_{z,w,n}^\infty=\{e_\infty,e_w^+,e_z^-\}\cup\{e_x^+,e_x^-\big|x\in\cl(Y)\text{ and }x>w\text{ or }x­<z\}$ where $z,w$ with $z<w$ range over all $Y$ and $n\in\mathbb N$ form a neighborhood basis of $e_\infty$.
\item For $z\in Y$, the sets $\hat{U}_{z,w,n}=\{e_z^+,e_w^-\}\cup\{e_x^+,e_x^-\big|x\in\cl(Y)\text{ and }z<x<w\}$ where $n\in\mathbb N$ and $w\in Y$ with $z<w$ form a neighborhood basis of $e_z^+$.
\item For $z\in Y$, the sets $\hat{U}_{w,z,n}$ where $n\in\mathbb N$ and $w\in Y$ with $w<z$ form a neighborhood basis of $e_z^-$.
\item For $x\in\cl(Y)\setminus Y$, the sets $\hat{U}_{w,z,n}$ where $n\in\mathbb N$ and $w,z\in Y$ with $w<x<z$ form a neighborhood basis of $e_x=e_x^+=e_x^-$.
\item If $z,w\in\cl(Y)$ are such that $z<w$ and there are no further points of $Y$ between $z$ and $w$, then $e_z^+=e_w^-$.
\item If $\cl(Y)$ has a maximum element, $z$, then $e_z^+=e_\infty$.
\item If $\cl(Y)$ has a minimum element, $z$, then $e_z^-=e_\infty$.
\end{itemize}

The above provides a description of $\mathcal E(P_Y(H))$ for a given $Y$. It is easy to see that, regardless of the choice of $Y$, each of the neighborhood bases of points described above can be made countable by taking a suitable subset, and hence every end of $P_Y(H)$ is short.

For an example, if $Y$ is the whole $y$-axis, then the resulting manifold with boundary $P_Y(H)$ (sometimes known as {\em the} Pr\"ufer manifold) has an end space similar to a double real line. One can picture it as two copies of $\mathbb R$, denoted $\mathbb R_+$ and $\mathbb R_-$, where neighborhoods of an $x_+\in\mathbb R_+$ are the union of a half-open interval $[x,z)\subseteq\mathbb R_+$ together with an open interval $(x,z)\subseteq\mathbb R_-$, and neighborhoods of an $x_-\in\mathbb R_-$ are the union of a half-open interval $(z,x]\subseteq\mathbb R_-$ together with an open interval $(z,x)\subseteq\mathbb R_+$; along with a further point at infinity.


As we have remarked, the Pr\"uferization procedure produces a manifold with boundary rather than a manifold. There are two ways to get around this issue and produce an honest manifold. The first one is to attach, instead of a copy of the real line to each element of $Y$, a copy of a closed half-plane (equivalently, one attaches the copies of the real line as described above and then collars the boundary of the resulting bordered manifold). This construction can be found, for example, in~\cite[Appendix A, pp. 466--477]{Spivak} or in~\cite[Example 1.3.1]{HubbardTeichmuller1} (to the authors' knowledge, this construction is originally due to Rad\'o~\cite{RadoUber}). The problem with this construction is that the analogs of the sets $U_{z,w,n}$ no longer separate the manifold. As a result of that, the manifold obtained with this process is always one ended, regardless of the subset $Y$ of the $y$-axis employed for the construction. If the set $Y$ utilized to carry out this procedure is countable, then the resulting manifold $M$ is metrizable; moreover, the unique end of $M$ is short and hence \( M \) is homeomorphic to \( \br^2 \). On the other hand, if $Y$ is uncountable then the resulting manifold $M$ is not only non-metrizable, but it fails to be separable as well; in this case, the unique end of this manifold is neither short nor long.

The other possibility, that produces a more interesting end space, is to take a manifold $P_Y(H)$ and ``double'' it---take two copies of $P_Y(H)$ and identify each of the boundary components corresponding to a given $y\in Y$ in both copies. The reader may verify that the double of $P_Y(H)$ has the same endspace as $P_Y(H)$---basic open neighborhoods are given by ``pairs'' of $\hat{U}_{z,w,n}$, one on each of the two copies of $P_Y(H)$.

We now consider a process, called Mooreization, of a manifold obtained by Pr\"uferization. We take a manifold of the form $P_Y(H)$ and, given a subset $Z\subseteq Y$, for each $z\in Z$ we take the corresponding copy of the real line $\mathbb R_z$ and, within it, identify each $x$ with $-x$; the resulting manifold will be denoted by $\hat{M}(Y,Z)$. Similar to the collaring process, this prevents neighborhoods of the form $U_{z,w,n}$ from separating the manifold whenever $z\in Z$ or $w\in Z$. As a result of this, such a manifold will have the same end space---but it will not be the same manifold---as $P_{Y\setminus Z}(Y)$, the result of having Pr\"uferized only the points in $Y\setminus Z$. The same end space is obtained if one doubles the manifold \( \hat M(Y,Z) \)---note that, in \( \hat M(Y,Z) \), the copies of $\mathbb R$ that correspond to each point $z\in Z$ are no longer boundary components after the identification.
The result of doubling \( \hat M(Y,Z) \) along its boundary components will be denoted by $M(Y,Z)$. The reader should note that, if one carries out this construction with an uncountable $Y$, irrespective of the subset $Z\subseteq Y$, then the manifold just described will not be metrizable, as choosing a point from each of the $\mathbb R^+_z$ ($z\in Y$) yields an uncountable discrete set (which prevents the manifold from being second countable). This leads to the main example of this appendix.

\begin{Thm}\label{thm:counterexample}
There exists a non-metrizable manifold $M$, with all ends short, such that $\mathcal E(M)$ is second countable. Moreover, $\mathcal E(M)$ is a Cantor space.
\end{Thm}

\begin{proof}
Let $Y$ be the whole $y$-axis, and let $Z\subseteq Y$ be the set of all irrational numbers. Pr\"uferize each point in $Y$ and then Mooreize each point in $Z$, before doubling along the boundary to obtain the manifold $M=M(Y,Z)=M(\mathbb R,\mathbb R\setminus\mathbb Q)$. Then---as explained in the previous paragraph---$M$ will fail to be metrizable as $Z$ (and consequently also $Y$) is uncountable. The topology of the end space $\mathcal E(M)$ has a basis consisting of the sets $\hat{U}_{z,w,n}$ and $\hat{U}_{z,w,n}^\infty$, where $z,w\in\mathbb Q$ and $z<w$; thus, $\mathcal E(M)$ is second countable. As $\mathcal E(M)$ is also compact and Hausdorff
, it follows that it is metrizable (e.g. by~\cite[Theorem 4.2.8]{Engelking}). The reader may easily verify that $\mathcal E(M)$ has no isolated points; moreover it is zero-dimensional (
by Proposition~\ref{prop:ends}), compact and metrizable, and hence it is in fact homeomorphic to the Cantor set by Brouwer's theorem (see e.g.~\cite[Exercise 6.2.A (c)]{Engelking}).
\end{proof}

\bibliographystyle{amsplain}
\bibliography{references}

\providecommand{\bysame}{\leavevmode\hbox to3em{\hrulefill}\thinspace}
\providecommand{\MR}{\relax\ifhmode\unskip\space\fi MR }
\providecommand{\MRhref}[2]{%
  \href{http://www.ams.org/mathscinet-getitem?mr=#1}{#2}
}
\providecommand{\href}[2]{#2}
\begin{thebibliography}{10}

\bibitem{AhlforsRiemann}
Lars Ahlfors and Leo Sario, \emph{Riemann surfaces}, Princeton University
  Press, 1960.

\bibitem{GauldMappingClass}
Mathieu Baillif, Satya Deo, and David Gauld, \emph{The mapping class group of
  powers of the long ray and other non-metrisable spaces}, Top. Appl.
  \textbf{157} (2010), 1314--1324.

\bibitem{BrownLocally}
Morton Brown, \emph{Locally flat imbeddings of topological manifolds}, Ann. of
  Math. (2) \textbf{75} (1962), 331--341. \MR{133812}

\bibitem{DevlinConstructibility}
Keith~J. Devlin, \emph{Constructibility}, Springer, Berlin, 1984.

\bibitem{DickmanUniform}
R.~F. Dickman, Jr., \emph{Some characterizations of the {Freudenthal}
  compactification of a semicompact space}, Proc. Amer. Math. Soc. \textbf{19}
  (1968), 631--633.

\bibitem{Engelking}
Ryszard Engelking, \emph{General topology}, 2 ed., Sigma {Series} in Pure
  Mathematics, vol.~6, Heldermann, Berlin, 1989.

\bibitem{FreedmanTopology}
Michael~H Freedman and Frank Quinn, \emph{Topology of 4-manifolds (pms-39)},
  vol.~39, Princeton University Press, 2014.

\bibitem{Freudenthal}
Hans Freudenthal, \emph{{\"Uber} die {Enden} topologischer {R\"aume} und
  gruppen}, Springer, Berlin, 1931.

\bibitem{GauldNonmetrisable}
David Gauld, \emph{Non-metrisable manifolds}, Springer, Singapore, 2014.

\bibitem{GreenwoodThesis}
Sina Greenwood, \emph{Nonmetrisable manifolds}, Ph.D. thesis, University of
  Auckland, 1999.

\bibitem{GreenwoodTrees}
\bysame, \emph{Constructing {Type} {I} nonmetrisable manifolds with given
  {$\Upsilon$}-trees}, Top. Appl. \textbf{123} (2002), 91--101.

\bibitem{HubbardTeichmuller1}
John~Hamal Hubbard, \emph{Teichm\"{u}ller theory and applications to geometry,
  topology, and dynamics. {V}ol. 1}, Matrix Editions, Ithaca, NY, 2006.

\bibitem{Kerekjarto}
B{\'e}la Ker{\'e}kj{\'a}rt{\'o}, \emph{Vorlesungen {\"u}ber topologie}, vol.~I,
  Springer, Berlin, 1923.

\bibitem{KunenSetTheory}
Kenneth Kunen, \emph{Set theory. an introduction to independence proofs.},
  North Holland, Amsterdam, 1992.

\bibitem{NyikosTheory}
Peter Nyikos, \emph{The theory of nonmetrizable manifolds}, Handbook of
  set-theoretic topology, North-Holland, Amsterdam, 1984, pp.~633--684.
  \MR{776633}

\bibitem{PorterWoods}
Jack~R. Porter and R.~Grant Woods, \emph{Extensions and absolutes of
  {Hausdorff} spaces}, Springer, New York, 1988.

\bibitem{RadoUber}
Tibor Rad{\'o}, \emph{{\"U}ber den begriff der {Riemannschen} fl{\"a}che}, Acta
  Litt. Sci. Szeged \textbf{2} (1925), no.~101-121, 10.

\bibitem{RichardsClassification}
Ian Richards, \emph{On the classification of noncompact surfaces}, Trans. Amer.
  Math. Soc. \textbf{106} (1963), 259--269.

\bibitem{Spivak}
Michael Spivak, \emph{A comprehensive introduction to differential geometry},
  3rd. ed., vol.~1, Publish or Perish, Houston, Texas, 1999.

\bibitem{Willard}
Stephen Willard, \emph{General topology}, Addison-Wesley, 1970.

\end{thebibliography}

\end{document}